\documentclass[11pt]{amsart}

\usepackage{epsfig, color}

\usepackage{amsmath}
\usepackage{amsthm}
\usepackage{hyperref}
\usepackage[margin=1.0in]{geometry}

\numberwithin{equation}{section}

\newtheorem{prop}{Proposition}
\newtheorem{lemma}[prop]{Lemma}

\newtheorem{thm}[prop]{Theorem}
\newtheorem{cor}[prop]{Corollary}

\numberwithin{prop}{section}

\theoremstyle{definition}
\newtheorem{defn}[prop]{Definition}

\newtheorem{rmk}[prop]{Remark}

\newcommand{\del}{\partial}
\newcommand{\delb}{\bar{\partial}}\newcommand{\dt}{\frac{\partial}{\partial t}}
\newcommand{\brs}[1]{\left| #1 \right|}

\newcommand{\gd}{\delta}

\newcommand{\gl}{\lambda}
\newcommand{\gL}{\Lambda}
\newcommand{\gU}{\Upsilon}

\newcommand{\gw}{\omega}
\newcommand{\ga}{\alpha}
\newcommand{\gb}{\beta}

\newcommand{\N}{\nabla}

\newcommand{\HH}{\mathcal H}

\newcommand{\EE}{\mathcal E}
\newcommand{\LL}{\mathcal L}
\newcommand{\til}[1]{\widetilde{#1}}
\newcommand{\nm}[2]{\brs{\brs{ #1}}_{#2}}

\renewcommand{\bar}[1]{\overline{#1}}

\renewcommand{\i}{\sqrt{-1}}

\newcommand{\bl}{\bar{l}}

\newcommand{\bw}{\bar{w}}

\newcommand{\bz}{\bar{z}}
\newcommand{\IP}[1]{\left<#1\right>}

\newcommand{\bga}{\bar{\alpha}}

\DeclareMathOperator{\osc}{osc}

\DeclareMathOperator{\Id}{Id}

\DeclareMathOperator{\End}{End}

\DeclareMathOperator{\Real}{Re}

\begin{document}

\title{Evans-Krylov Estimates for a nonconvex Monge Amp\`ere equation}

\date{\today}

\begin{abstract} We establish Evans-Krylov estimates for certain nonconvex fully
nonlinear elliptic and parabolic equations by exploiting partial Legendre
transformations.  The equations under consideration arise in part from the study
of the ``pluriclosed flow'' introduced by the first author and Tian \cite{ST1}. 
\end{abstract}

\author{Jeffrey Streets}
\address{Rowland Hall\\
         University of California, Irvine\\
         Irvine, CA 92617}
\email{\href{mailto:jstreets@uci.edu}{jstreets@uci.edu}}
\thanks{J. Streets and M. Warren gratefully acknowledge support from the NSF via DMS-1301864 and DMS-1161498, respectively. }

\author{Micah Warren}
\address{Fenton Hall\\
         University of Oregon\\
         Eugene, OR 97403}
\email{\href{mailto:micahw@uoregon.edu}{micahw@uoregon.edu}}

\maketitle

\section{Introduction}
\subsection{Statement of Main Estimate}

Consider $U \subset \mathbb R^k \times \mathbb R^l$ with coordinates
$\{x_i\}_{i=1}^k$ and $\{y_i\}_{i=1}^l$, and let $u \in C^{\infty}( U)$ be
convex in the $x$ variables and concave in the $y$ variables.  We
will consider two equations in this setting.  First, we have the \emph{real
twisted Monge-Amp\`ere equation}
\begin{align} \label{realelliptic}
F(u) := \log \det u_{x_i x_j} - \log \det (- u_{y_i y_j}) = 0.
\end{align}
Next we consider the \emph{parabolic real twisted Monge-Amp\`ere equation}
\begin{align} \label{realparabolic}
H(u) := \dt u - F(u) = 0.
\end{align}
We also consider similar equations using complex variables.  In particular,
consider $U \subset \mathbb C^k \times \mathbb C^l$ with coordinates
$\{z_i\}_{i=1}^k$ and $\{w_i \}_{i=1}^l$, and let $u \in C^{\infty}(U)$ be
plurisubharmonic in the $z$ variables and plurisuperharmonic in the $w$
variables.  We
have the \emph{complex twisted Monge-Amp\`ere equation}
\begin{align} \label{complexelliptic}
F_{\mathbb C}(u) := \log \det u_{z_i \bz_j} - \log \det (- u_{w_i \bw_j}) = 0.
\end{align}
Lastly we have the \emph{parabolic complex twisted Monge-Amp\`ere equation},
\begin{align} \label{complexparabolic}
H_{\mathbb C}(u) := \dt u - F_{\mathbb C}(u) = 0.
\end{align}
With uniform convexity assumptions made (see Definition \ref{conedef}),
equations
(\ref{realelliptic}) and (\ref{complexelliptic}) are uniformly elliptic,
whilst
(\ref{realparabolic}) and (\ref{complexparabolic}) are uniformly parabolic. 
However,
these equations are \emph{neither convex nor concave}, and as such the
Evans-Krylov theory does not apply to these equations. 

The celebrated result of Evans and Krylov states that for uniformly elliptic (or
parabolic) equations, one can conclude interior H\"{o}lder estimates on the
second derivatives from $C^{1,1}$ estimates, provided the equation is convex or
concave.  There are a few more general cases where such a statement can be made:
Caffarelli and Cabre \cite{CC2} showed results in the case where the functional
$F$ is a minimum of convex and concave functions.   Caffarelli and Yuan
\cite{CY} demonstrate the estimates under a partial convexity condition. Yuan
\cite{Y01} has proved such estimates in the specific case of the $3$-dimensional
special Lagrangian equations.  In full generality these estimates are known to
fail for nonconvex equations : Nadirashvilli and Vl\u{a}du\c {t} \cite{NV}
exhibit a 2-homogeneous function satisfying a uniformly elliptic equation, which
is $C^{1,1}$ and not  $C^2$ at the origin.  

The main purpose of this paper is to establish a $C^{2,\alpha}$ estimate for solutions of twisted Monge Amp\`ere equations from uniform bounds on the Hessian.

\begin{thm} \label{EvansKrylov} Given $k,l \in \mathbb N$ and $\gl,\gL> 0$ there
exists $C = C(k,l,\gl,\gL)$ and $\ga = \ga(k,l,\gl,\gL) > 0$ such that
\begin{itemize}
\item If $u \in \EE^{k,l,\mathbb R }_{B_2,\gl,\gL}$ is a solution to
(\ref{realelliptic})
then $\nm{u}{C^{2,\ga}(B_1)} \leq C$.
\item If $u \in \EE^{k,l,\mathbb R}_{Q_2,\gl,\gL}$ is a solution to
(\ref{realparabolic})
then $\nm{\dt u}{C^{\ga}(Q_1)} + \nm{u}{C^{2,\ga}(Q_1)} \leq C$.
\item If $u \in \EE^{k,l,\mathbb C}_{B_2,\gl,\gL}$ is a solution to
(\ref{complexelliptic}) then $\nm{u}{C^{2,\ga}(B_1)} \leq C$.
\item If $u \in \EE^{k,l,\mathbb C}_{Q_2,\gl,\gL}$ is a solution to
(\ref{complexparabolic}) then $\nm{\dt u}{C^{\ga}(Q_1)} + \nm{u}{C^{2,\ga}(Q_1)}
\leq C$.
\end{itemize}
\end{thm}

The key insight to establish these estimates comes from applying a partial
Legendre
transformation in the real case.   More specifically, as the functions in
question are mixed
concave/convex, we can apply a Legendre transformation to the concave variables
to obtain a strictly convex function.  This transformation was described by
Darboux \cite{Darboux}, who observed how it has the effect of linearizing the
two-dimensional real Monge-Amp\`ere equation.  This was exploited to establish
regularity properties for the $2$-dimensional Monge Amp\`ere equation by many
authors, and later in general dimension \cite{RSW}.  For the classical
theory, see 
\cite{HW,H1,H2,H3,Schulz2,Schulz} and for more recent applications see
\cite{GuanPhong}.  In Proposition \ref{translaw} below
we show that more generally one can transform solutions to (\ref{realelliptic})
 into the real elliptic 
Monge-Amp\`ere equation.  This observation alone suffices to establish an
analogue
of the rigidity result of Calabi-Jorgens-Pogorelov (\cite{Calabi, Jorgens,
Pogorelov}) for uniformly concave solutions of (\ref{realelliptic})

This result can be used in conjunction with a blowup argument to establish the
first two claims of Theorem \ref{EvansKrylov}.  However, even with the natural
hypotheses
for equations (\ref{complexelliptic}) and (\ref{complexparabolic}) of
plurisub/superharmonicity, there is no well-defined version of a complex 
Legendre transformation. Thus this method alone cannot establish
Theorem \ref{EvansKrylov} in these cases.  Nonetheless, we use the
transformation law to take quantities which are subsolutions to parabolic
equations associated to the complex parabolic Monge-Amp\`ere equation, and
express
them in terms of the inverse Legendre-transformed coordinates, assuming this
transformation were defined.  As it turns out, these quantities in the original
coordinates are \emph{always} defined (irrespective of whether the Legendre
transformation is defined), and still are subsolutions of certain parabolic
equations.  This key observation can be exploited to adapt the usual proof of
the Evans-Krylov 
estimate for convex equations to our setting, establishing Theorem
\ref{EvansKrylov}.

\subsection{Consequences for pluriclosed flow}

In \cite{SPCFSTB,ST1,ST2,STGK} the second author and Tian introduced and
developed a
geometric flow of pluriclosed metrics on complex manifolds.  Briefly, given
$(M^{2n}, g_0, J)$ a Hermitian manifold such that the associated K\"ahler form
$\gw_0$ is pluriclosed, i.e.
$\del\delb \gw_0 = 0$, we say that a one-parameter family $g_t$ of metrics 
metrics is a solution to \emph{pluriclosed flow} if the associated K\"ahler
forms satisfy
\begin{align} \label{PCF}
\dt \gw =&\ \del \del^*_{\gw} \gw + \delb \delb^*_{\gw} \gw + \frac{\i}{2} \del
\delb \log \det g.
\end{align}
This evolution equation is strictly parabolic and preserves the pluriclosed
condition.  Moreover, in \cite{STGK} it was discovered that pluriclosed flow
preserves generalized K\"ahler geometry in an appropriate sense.  Recall that a
generalized K\"ahler manifold is a quadruple $(M^{2n}, g, J_{A}, J_B)$
consisting
of a smooth manifold with two integrable complex structures $J_{A}, J_B$, a
metric
$g$ which is compatible with both, and moreover satisfies the conditions
\begin{align*}
d^c_A \gw_A = - d^c_B \gw_B, \qquad d d^c_A \gw_A = 0.
\end{align*}
If we impose the further condition that $[J_A,J_B] = 0$, one obtains an
integrable splitting of $T_{\mathbb C} M$, and moreover the pluriclosed flow in
this setting reduces (up to background terms) to the parabolic complex twisted
Monge Amp\`ere equation (\cite{SPCFSTB} Theorem 1.1).  This is expounded upon in
\S \ref{PCFsec}.  Combining Theorem \ref{EvansKrylov} with a blowup argument
yields higher order regularity of the flow in
the presence of uniform metric estimates.  This result plays a key role in
establishing new long time existence and convergence results for the pluriclosed
flow, and these will be described in a future work.

\begin{thm} \label{PCFEK} Let $(M^{2n}, g_0, J_A, J_B)$ be a compact generalized
K\"ahler manifold satisfying $[J_A,J_B] = 0$.  Let $g_t$ be the solution to
pluriclosed flow with initial condition $g_0$.  Suppose the solution exists on
$[0,\tau)$, $\tau < \tau^*(g)$ (see Definition \ref{taustardef}), and there
exists constants $\gl, \gL$ such that
\begin{align*}
 \gl g_0 \leq g_t \leq \gL g_0.
\end{align*}
Given $k \in \mathbb N$, $\ga \in [0,1)$ there exists a constant $C =
C(k,\ga,g_0, \gl,\gL,\tau)$ such that
\begin{align*}
\nm{g_t}{C^{k,\ga}} \leq C. 
\end{align*}
\end{thm}

\subsection{Outline}

In \S \ref{legbck} we recall the
partial Legendre transformation and determine transformation laws for the PDEs
in question.  Inspired by these transformation laws, in \S \ref{evolsec} we
state monotonicity formulas for certain combinations of second derivatives along
solutions to (\ref{realparabolic}), (\ref{complexparabolic}).  As the proofs
consist of lengthy, tedious calculations we relegate them to an appendix, \S
\ref{calcsec}.  Using this key input we establish Theorem \ref{EvansKrylov} in
\S \ref{eksec}.  In \S \ref{PCFsec}  we recall how to reduce solutions to the
pluriclosed flow on commuting generalized K\"ahler manifolds to solutions of a
scalar PDE
which reduces to (\ref{complexparabolic}) on flat space, and then use Theorem
\ref{EvansKrylov} to establish Theorem \ref{PCFEK}.

\section{Background on Legendre Transformation} \label{legbck}
\subsection{Real Legendre Transformation} \label{realleg}

We briefly recall the Legendre transformation and one of its key properties for
us.  For a smooth convex function $u : \mathbb R^n \to \mathbb R$, the
\emph{Legendre transformation} is defined first by the change of variables
\begin{align*}
y_i(x) =&\ \frac{\del u}{\del x_i}(x),
\end{align*}
then by declaring
\begin{align*}
w(y) =&\ x_i \frac{\del u}{\del x_i}(x) - u(x).
\end{align*}
Observe that the Jacobian of the coordinate change takes the form
\begin{align*}
\frac{\del y_i}{\del x_j} =&\ \frac{\del^2 u(x)}{\del x_i x_j}.
\end{align*}
On the other hand, the Legendre transformation is involutive, so it follows that
\begin{align*}
\frac{\del x_i}{\del y_j} =&\ \frac{\del^2 w(y)}{\del y_i \del y_j}.
\end{align*}
Thus we see that the Legendre transformation ``inverts the Monge Amp\`ere
operator" in the sense that
\begin{align*}
\det \frac{\del^2 w}{\del y_i \del y_j} = \left( \det \frac{\del^2
u}{\del x_i \del x_j} \right)^{-1}.
\end{align*}
This basic fact lies at the heart of our constructions below.

\subsection{Transformation Laws for PDE and Rigidity Results}

In this subsection we build on the observation in \S \ref{realleg} of the
transformation law for the Monge Amp\`ere operator under Legendre transformation
to define a partial Legendre transformation which can convert the twisted Monge
Amp\`ere equation into the usual Monge Amp\`ere equation.  To begin we define a
class of functions with certain convexity hypotheses.

\begin{defn} \label{conedef} Given $U \subset \mathbb R^k \times \mathbb R^l$
and constants $\gl,\gL > 0$ we
set
\begin{align*}
\EE^{k,l,\mathbb R}_{U} :=&\ \{ u \in C^{\infty}(U)\ | \ \left. D^2 u
\right|_{\mathbb
R^k}
> 0 ,
\quad \left. D^2 u \right|_{\mathbb R^l} < 0 \},\\
\EE^{k,l,\mathbb R}_{U,\gl,\gL} :=&\ \{ u \in \EE^{k,l,\mathbb R}_{U}\ |\ \gl
I_k \leq \left. D^2 u
\right|_{\mathbb R^k} \leq \gL I_k, \quad \gl I_l \leq (- \left. D^2 u
\right|_{\mathbb R^l}) \leq \gL I_l \}.
\end{align*}
Similarly, given $U \subset \mathbb C^k \times \mathbb C^l$
\begin{align*}
\EE^{k,l,\mathbb C}_{U} :=&\ \{ u \in C^{\infty}(U)\ | \ \left. \i \del\delb u
\right|_{\mathbb
C^k} > 0 , \quad \left. \i \del\delb u \right|_{\mathbb C^l} < 0 \},\\
\EE^{k,l,\mathbb C}_{U,\gl,\gL} :=&\ \{ u \in \EE^{k,l, \mathbb C}_{U}\ |\ \gl
I_k \leq \left.
\i\del\delb u
\right|_{ \mathbb C^k} \leq \gL I_k, \quad \gl I_l \leq (- \left. \i \del\delb u
\right|_{\mathbb C^l}) \leq \gL I_l \}.
\end{align*}
\end{defn}

\begin{defn} \label{PLdef} Given $u \in \EE_{\Omega}^{k,l,\mathbb R}$, the
\emph{partial
Legendre
transformation} is the map
\begin{align*}
\mathcal{PL}_{k,l} : \EE^{k,l,\mathbb R}_{\Omega} \to \EE^{n,0,\mathbb
R}_{\hat\Omega},
\end{align*}
defined as follows.  Let $(x,y)$ denote 
coordinates on the domain $\Omega \subset \mathbb R^k \times \mathbb R^l$ where
the function is originally defined.  Let $(x,z)$ denote coordinates of the set
$\hat{\Omega}$ that is the image of $\Omega$ under the map 
\begin{align}
(x,y) \longmapsto \left(x , \frac{\del u}{\del y}(x,y)\right).  
\end{align}
Now on $\hat{\Omega},$ define

\begin{align} \label{pltrans}
y(x,z) = \left\{ y\ |\ \frac{\del u}{\del y}(x,y) = z \right\}.
\end{align}
Then we have $\mathcal{PL}_{k,l}(u) = w(x,z)$, where
\begin{align*}
w(x,z) = u(x,y(x,z)) - \IP{y(x,z),z}.
\end{align*}
\end{defn}

\begin{prop} \label{translaw} (Transformation law for twisted Monge Amp\`ere)
Given $u \in
\EE^{k,l,\mathbb R}$, one has
\begin{align} \label{Wcalc}
 D^2 \mathcal{PL}_{k,l} u =&\ \left(
\begin{matrix}
 u_{xx} - u_{xy} u_{yy}^{-1}  u_{yx} & u_{yx} u_{yy}^{-1}\\
u_{yy}^{-1} u_{yx} & - u_{yy}^{-1}
\end{matrix}\right) \geq 0.
\end{align}
Moreover,
\begin{align} \label{dettransform}
\det D^2 \mathcal{PL}_{k,l} u = \frac{ \left. \det D^2 u \right|_{\mathbb
R^k}}{\det \left( \left. - D^2 u \right|_{\mathbb R^l} \right)}.
\end{align}
\begin{proof} Let $\Phi(x,y) = (x,z)$ denote the coordinate transformation
defined by (\ref{pltrans}).  The defining equation (\ref{pltrans}) can be
rewritten as $y = D_y u(x,z)$.  Differentiating this yields
\begin{align*}
I =&\ \frac{\del y}{\del y} = D_y D_y u(x,z) \frac{\del z}{\del y},
\end{align*}
thus $\frac{\del z}{\del y} = \left(D_y D_y u\right)^{-1}$.  Also
\begin{align*}
 0 =&\ \frac{\del y}{\del x} = D_x D_y u + D_y D_y u \frac{\del z}{\del x},
\end{align*}
thus $\frac{\del z}{\del x} = - (D_y D_y u)^{-1} D_x D_y u$.  Thus
\begin{align} \label{Tdef}
 D \Phi =&\ \left( 
 \begin{matrix}
  I_k & 0\\
  - (D_y D_y u)^{-1} D_x D_y u & (D_y D_y u)^{-1}
 \end{matrix} \right).
\end{align}
Using this, a further direct calculation shows
\begin{align*}
 \N w =&\ \left( u_x, - z \right).
\end{align*}
Thus
\begin{align} \label{twistedhessian}
 D^2 w =&\ \left(
 \begin{matrix}
  u_{xx} + u_{yx} \cdot \frac{\del z}{\del x} & u_{xy} \frac{\del z}{\del y}\\
  - \frac{\del z}{\del x} & - \frac{\del z}{\del y}
 \end{matrix}
\right) = \left( \begin{matrix}
 u_{xx} - u_{xy} u_{yy}^{-1}  u_{yx} & u_{yx} u_{yy}^{-1}\\
u_{yy}^{-1} u_{yx} & - u_{yy}^{-1}
\end{matrix}\right),
\end{align}
establishing (\ref{Wcalc}).  The formula (\ref{dettransform}) follows from the
block determinant formula
\begin{align*}
\det \left(
\begin{matrix}
A & B\\
C & D
\end{matrix} \right) = \det D \det \left(A - B D^{-1} C \right).
\end{align*}

To conclude the matrix is nonnegative, write it symbolically  as
\[%
\begin{pmatrix}
A+BCB^{T} & -BC\\
-C^{T}B^{T} & C
\end{pmatrix}
.
\]
Then one can easily check that
\begin{align*}
&
\begin{pmatrix}
A+BCB^{T} & -BC\\
-C^{T}B^{T} & C
\end{pmatrix}%
\begin{pmatrix}
\vec{v}\\
\vec{w}%
\end{pmatrix}
\cdot%
\begin{pmatrix}
\vec{v}\\
\vec{w}%
\end{pmatrix}
\\
& =A\vec{v}\cdot\vec{v}+C\left(  \vec{w}-B^{T}\vec{v}\right)  \cdot\left(
\vec{w}-B^{T}\vec{v}\right)  \geq 0.  
\end{align*}

\end{proof}
\end{prop}

\begin{cor} \label{rigidity} Given $\gl,\gL > 0$, let $u \in \EE^{k,l}_{\mathbb
R^n,\gl,\gL}$ be a solution
to (\ref{realelliptic}).  Then $u$ is a quadratic polynomial.
\begin{proof} Given a lower bound on the concavity in the second variable, we
conclude the coordinate $z$ is global, for each $x$. Directly using
(\ref{Wcalc}) we see that 
$\mathcal{PL}_{k,l}$ is an entire convex function.   Also, using
(\ref{dettransform}) and (\ref{realelliptic}) we conclude that
\begin{align*}
 \det D^2 \mathcal{PL}_{k,l} u = 1.
\end{align*}
It follows from (\cite{Calabi, Jorgens, Pogorelov}) that $\mathcal {PL}_{k,l} u$
is a quadratic polynomial, and its Hessian is constant.  Comparing against
(\ref{Wcalc}), we conclude that $u_{yy}^{-1}$ is constant, thus $u_{xy}$ is
constant, and then finally $u_{xx}$ is constant.  Thus $u$ is a quadratic
polynomial. 
\end{proof}
\end{cor}

\begin{rmk} The direct Legendre transformation method of Corollary
\ref{rigidity} gives a Liouville-Bernstein type property for solutions to
(\ref{realelliptic}) which are uniformly concave in the second variable.  We
will give a second proof of Evans-Krylov type
regularity in \S \ref{evolsec}-\ref{eksec} which does not actually require
direct use of the Legendre transformation, but rather makes estimates directly
on the matrix (\ref{Wcalc}), which is always defined in the original
coordinates, and
has the crucial subsolution property necessary in the proof of the Evans-Krylov
estimate.  This has the key advantage of extending to the complex settings,
where plurisubharmonicity does not suffice to define a genuine Legendre
transform.
\end{rmk}

\subsection{Complex Legendre Transformation} \label{compleg}

While convexity hypotheses are natural for understanding the real Monge Amp\`ere
equation, the natural hypothesis to impose for complex Monge-Amp\`ere type
equations is plurisubharmonicity.  Many Legendre-type transformations have been
proposed for functions on $\mathbb C^n$ satisfying conditions weaker than
convexity, see for instance (\cite{Dem,Hormander,Kisel,Lempert}).  More
recently there is a proposal (\cite{Bernd}) to define a Legendre transform for
plurisubharmonic functions asymptotically, using Bergman kernels.  

Now by direct analogy, if a complex Legendre transform where defined, the
complex Hessian of the transform would be the inverse of the complex Hessian of
the function, and we could argue as above to apply known results about the
complex Monge-Amp\`ere equation to the twisted equation. While we cannot find
such a function, we operate directly on the partially transformed complex
Hessian, proceeding obliviously as if the transform where defined.  This yields
a number of nontrivial maximum
principles for (\ref{complexelliptic}) and (\ref{complexparabolic}),  which
would
otherwise be difficult to discover or motivate.

\section{Evolution Equations} \label{evolsec}
\subsection{Real Case} \label{realevol}

As explained in \S \ref{legbck} we will now derive a priori estimates for
(\ref{realelliptic})-(\ref{complexparabolic}) without the explicit use of the
Legendre transformation.  The basic idea is to think of the Legendre
transformation ``infinitesimally'' and use it as a change of variables on the
tangent space which helps us identify the right quantities/linearized operators
which have favorable maximum principles.  The first step is to recall that, in
the context of the ``pure'' Monge Amp\`ere equation for a function $w$, the
linearized operator is
\begin{align*}
\mathcal L =  w^{\ga \gb}
\frac{\del^2}{\del x^{\ga} \del x^{\gb}}.
\end{align*}
We now rewrite this operator using the partial Legendre transformation.

\begin{lemma}  Let $u \in \EE^{k,l}$, and $w = \mathcal{PL}_{k,l} u$. 
Furthermore let $T = D \Phi$ denote the change of basis matrix for the Legendre
transformation as in (\ref{Tdef}).
\begin{align*}
 L^{ab} = w^{\ga \gb} T_{\ga}^a T_{\gb}^b.
\end{align*}
Then
\begin{align*}
 L =&\ \left( \begin{matrix}
               u_{xx}^{-1} & 0\\
               0 & - u_{yy}^{-1}
              \end{matrix}
\right).
\end{align*}
\begin{proof} We emphasize that this formula is describing the matrix $L$ in
terms of the basis vectors $\left\{\frac{\del}{\del x}, \frac{\del}{\del z}
\right\}$.  First note that a direct calculation shows that
\begin{align*}
 w^{-1} =&\ \left(
 \begin{matrix}
               u_{xx}^{-1} & u^{xx} u_{xy}\\
               u_{yx} u^{xx} & - u_{yy}^{-1} + u_{yx} u^{xx} u_{xy} 
              \end{matrix} \right).
\end{align*}
We now directly compute $L$ block by block.  First:
\begin{align*}
 L^{x_i x_j} =&\ w^{\ga\gb} T^{x_i}_{\ga} T^{x_j}_{\gb} = w^{x_k x_l}
T^{x_i}_{x_k} T^{x_j}_{x_l} = u^{x_k x_l} \gd_k^i \gd_l^j = u^{x_i x_j}.
\end{align*}
Next we have
\begin{align*}
 L^{x_i z_j} =&\ w^{\ga \gb} T^{x_i}_{\ga} T^{z_j}_{\gb}\\
 =&\ w^{x_k x_l} T^{x_i}_{x_k} T^{z_j}_{x_l} + w^{x_k y_l} T^{x_i}_{x_k}
T^{z_j}_{y_l} + w^{y_k x_l} T^{x_i}_{y_k} T^{z_j}_{x_l} + w^{y_k y_l}
T^{x_i}_{y_k} T^{z_j}_{y_l}\\
 =&\ u^{x_k x_l} \gd_k^i \left( - u^{y_j y_p} u_{y_p x_l} \right) + u^{x_k x_p}
u_{x_p y_l} \gd_k^i u^{y_j y_l} + 0 + 0\\
 =&\ 0.
\end{align*}
A similar calculation shows that $L^{z_i x_j} = 0$.  Lastly we have
\begin{align*}
 L^{z_i z_j} =&\ w^{\ga\gb} T^{z_i}_{\ga} T^{z_j}_{\gb}\\
 =&\ w^{x_k x_l} T^{z_i}_{x_k} T^{z_j}_{x_l} + w^{x_k y_l} T^{z_i}_{x_k}
T^{z_j}_{y_l} + w^{y_k x_l} T^{z_i}_{y_k} T^{z_j}_{x_l} + w^{y_k y_l}
T^{z_i}_{y_k} T^{z_j}_{y_l}\\
 =&\ u^{x_k x_l} (- u^{y_i y_p} u_{y_p x_k}) (- u^{y_j y_q} u_{y_q x_l}) +
u^{x_k x_p} u_{x_p y_l} (- u^{y_i y_q} u_{y_q x_k} )u^{y_j y_l}\\
&\ + u_{y_k x_p} u^{x_p x_l} u^{y_i y_k} (- u^{y_j y_p} u_{y_p x_l}) + \left( -
u_{y_k y_l} + u_{y_k x_p} u^{x_p x_q} u_{x_q y_l} \right) u^{y_i y_k} u^{y_j
u_l}\\
=&\ - u^{y_i y_j}.
 \end{align*}
\end{proof}
\end{lemma}

\begin{prop} \label{realsubsoln} Let $u \in \EE^{k,l,\mathbb R}$ be a solution
to
(\ref{realparabolic}). Let $W = D^2 \mathcal{PL}_{k,l} u$.  Then
\begin{align*}
\left(\dt - \LL \right) \frac{\del u}{\del t} =&\ 0,\\
\left(\dt - \LL \right) W \leq&\ 0.
\end{align*}
\begin{proof} See Lemma \ref{Wev}.
\end{proof}
\end{prop}

\subsection{Complex Case}

As discussed in \S \ref{legbck}, there is as yet not a clear Legendre
transformation defined for plurisubharmonic functions.  Nonetheless, we draw
inspiration from \S \ref{realevol} and obtain maximum principle estimates which
correspond to usual maximum principle estimates for the complex Monge Amp\`ere
equation,
assuming the Legendre transformation were defined.  In particular, assume $u \in
\EE^{k,l}$ and set
\begin{align} \label{cLdef}
 \LL := u^{\bz_b z_a} \frac{\del^2}{\del z_a \del \bz_b} - u^{\bw_b w_a}
\frac{\del^2}{\del w_a \del \bw_b}.
\end{align}
Also, define
\begin{align} \label{cW}
 W =&\ \left(
 \begin{matrix}
  u_{z_i \bz_j} - u_{z_i \bw_k} u^{\bw_k w_l} u_{w_l \bz_j} & u_{z_i \bw_k}
u^{\bw_k w_i}\\
 u^{\bw_j w_k} u_{w_k \bz_i} & - u^{\bw_j w_i}
 \end{matrix}
\right).
\end{align}
The form of this matrix is of course derived from the corresponding quantity in
(\ref{Wcalc}).  By a direct calculation one observes that equation
(\ref{complexparabolic}) is equivalent to
\begin{align} \label{complexparequiv}
 \dt u =&\ \log \det W.
\end{align}
As in the real case, the crucial input in obtaining Evans-Krylov
type estimates
for (\ref{complexparabolic}) and (\ref{complexelliptic}) is a subsolution
property for the matrix $W$.  The proposition below is simply a long calculation
which 
follows by elementary applications of the Cauchy-Schwarz inequality, and is
carried out in \S \ref{calcsec}.

\begin{prop} \label{complexsubsoln} Let $u \in \EE^{k,l,\mathbb C}$ be a
solution to
(\ref{complexparabolic}). Then
\begin{align*}
\left(\dt - \LL \right) \frac{\del u}{\del t} =&\ 0,\\
\left(\dt - \LL \right) W \leq&\ 0.
\end{align*}
\begin{proof} This follows directly from Lemmas \ref{complexevs} and
\ref{complexQsign}. 
\end{proof}
\end{prop}

\section{Evans-Krylov Type Estimates} \label{eksec}

In this section we establish a general oscillation estimate (Theorem \ref{genEK}
below) of Evans-Krylov type which replaces the
convexity hypothesis with its main implication in the original proofs of
Evans-Krylov: that the elliptic operator acts on a matrix of second partials
that is itself a 
subsolution of a uniformly elliptic equation in the matrix sense.  In
conjunction with the
subsolution properties arising from the partial Legendre transformation
(Propositions \ref{realsubsoln} and \ref{complexsubsoln}), we obtain Theorem
\ref{EvansKrylov} as an immediate corollary.  The proof is closely modeled after
(\cite{Lieb} Lemma 14.6).  We also refer the reader to the original works
(\cite{Evans}, \cite{Krylov}) as well as more recent versions of the proof
(\cite{CC}, \cite{CS}).  To begin we recall some standard notation and
results.

\begin{defn} Given $(w,s) \in \mathbb C^n \times \mathbb R$, let
\begin{align*}
 Q((w,s),R) := &\ \{ (z,t) \in \mathbb C^n \times \mathbb R |\ t \leq s, \quad
\max \{ \brs{z - w}, \brs{t - s}^{\frac{1}{2}} \} < R \},\\
 \Theta(R) :=&\ Q((w,s - 4 R^2),R).
\end{align*}
\end{defn}

\begin{thm} \label{pweakharnack} (\cite{Lieb} Theorem 7.37) Let $u$ be a
nonnegative function on $Q(4R)$ such that 
\begin{align*}
 - u_t + a^{ij} u_{ij} \leq 0,
\end{align*}
where
\begin{align}
 \gl \gd_i^j \leq a^{ij} \leq \gL \gd_i^j \label{lellipt}.
\end{align}
There are positive constants $C, p > 1$ depending only on $n,\gl,\gL$ such that
\begin{align} \label{weakharnack}
 \left( R^{-n-2} \int_{\Theta(R)} u^p \right)^{\frac{1}{p}} \leq C \inf_{Q(R)}
u.
\end{align}
\end{thm}

\begin{thm}
\label{genEK}Suppose that $u\in\mathcal{E}_{Q(R),\lambda,\Lambda}^{k,l}\cap
C^{2,1}(Q(R))$ \ satisfies
\[
\frac{\partial}{\partial t}u=F(W(D^{2}u))
\]
for $F$ a $\left(  \lambda,\Lambda\right)  $-elliptic functional and suppose
that the Hermitian matrix $W$ satisfies%
\[
\left(  \frac{\partial}{\partial t}-L\right)  W\leq0,
\]
for some $\left(  \lambda,\Lambda\right)  $-elliptic operator $L.$ \ Then
there are positive constants $\alpha,C$ depending only on $n,\lambda,\Lambda$
such that for all $\rho<R$,
\[
\osc_{Q(\rho)}u_{t}+\osc_{Q(\rho)}W\leq C(n,\lambda,\Lambda)\left(  \frac
{\rho}{R}\right)  ^{\alpha}\left(  \osc_{Q(R)}u_{t}+\osc_{Q(R)}W\right)  .
\]

\end{thm}

\begin{proof}
We are assuming that $F$ is $\left(  \lambda,\Lambda\right)  $-elliptic on a
convex set containing the range of $W.$ Thus for any two points $(x,t_{1}%
),(y,t_{2})$ $\in$ $Q(4R)$ there exists a matrix $a^{ij}$ satisfying
(\ref{lellipt}) such that

\begin{align} \label{conclogdet}
a^{ij}\left(  (x,t_{1}),(y,t_{2})\right)  \left(  W_{ij}(x,t_{1}%
)-W_{ij}(y,t_{2})\right)  =\frac{\partial u}{\partial t}(x,t_{1}%
)-\frac{\partial u}{\partial t}(y,t_{2}).
\end{align}
Now by (\cite{Lieb} Lemma 14.5)
we can choose a finite set of unit vectors $v_{\alpha}$ such that
\[
a^{ij}=\sum_{\alpha=1}^{N}f_{\alpha}v_{\alpha}^{i}\bar{v_{\alpha}^{j}}%
\]
for $f_{\alpha}$ depending on $(x,t_{1}),(y,t_{2})$ yet always satisfying
\[
f_{\alpha}\in\lbrack\lambda_{\ast},\Lambda_{\ast}],\text{ }%
\]
for some constants $\gl_*, \Lambda_*$ depending only on $\gl,\gL$.  Now let
$f_{0}=1$, and let
\begin{align*}
w_{0}:=  &  \ -\frac{\partial u}{\partial t}\\
w_{\alpha}:=  &  \ W_{v_{\alpha}\bar{v_{\alpha}}}.
\end{align*}
In this notation (\ref{conclogdet}) reads
\[
\sum_{\alpha=0}^{N\ }f_{\alpha}\left(  w_{\alpha}(x,t_{1})-w_{\alpha}%
(y,t_{2})\right)  =0.
\]
It follows that for every pair $(x,t_{1}),(y,t_{2})\in
Q(4R)$, we have
\[
f_{\alpha}\left(  w_{\alpha}(y,t_{2})-w_{\alpha}(x,t_{1})\right)  =\sum
_{\beta\neq\alpha}f_{\beta}\left(  w_{\beta}(x,t_{1})-w_{\beta}(y,t_{2}%
)\right)  .
\]
Now let
\[
M_{s\alpha}=\sup_{Q(sR)}w_{\alpha},\qquad m_{s\alpha}=\inf_{Q(sR)}w_{\alpha
},\qquad P(sR)=\sum_{\alpha}M_{s\alpha}-m_{s\alpha}.
\]
As each quantity $w_{\ga}$ is a subsolution to a uniformly parabolic equation,
it follows that $M_{4\ga} - w_{\ga}$ is a supersolution to a uniformly parabolic
equation, and hence by Theorem \ref{pweakharnack} we obtain
\begin{align*}
 \left( R^{-n-2} \int_{\Theta(R)} \left(M_{4 \ga} - w_{\ga} \right)^p
\right)^{\frac{1}{p}} \leq&\ C_1 \inf_{Q(R)}( M_{4 \ga} - w_{\ga} )= C_1 (M_{4
\ga} - M_{1 \ga}).
\end{align*}
Summing these inequalities and applying Minkowski's inequality yields, for any
fixed $\beta$,
\begin{gather} \label{Lieb367a}
\begin{split}
\left(  R^{-n-2}\int_{\Theta(R)}\sum_{\alpha\neq\beta}(M_{4\alpha}-w_{\alpha
})^{p}\right)  ^{\frac{1}{p}}  &  \leq \sum_{\alpha\neq\beta}\left(
R^{-n-2}\int_{\Theta(R)}\left(  M_{4\alpha}-w_{\alpha}\right)  ^{p}\right)
^{\frac{1}{p}}\\
&  \leq C_{1}\sum_{\alpha\neq\beta}M_{4\alpha}-M_{1\alpha}\\
&  \leq C_{1}\sum_{\alpha\neq\beta}M_{4\alpha}-m_{4\alpha}- \left(M_{1\alpha
}-m_{1\alpha} \right)\\
&  \leq C_{1}\left[  P(4R)-P(R)\right].
\end{split}
\end{gather}
Now choose $\left(x,t\right)  \in Q(4R)$ such that
\[
W_{\beta}\left(  x,t_{1}\right)  =m_{4\beta.}%
\]
We have
\[
\lambda_{\ast}\left(  W_{\beta}\left(  y,t\right)  -m_{4\beta}\right)  \leq
f_{\beta}\left(  y,t\right) \left(  W_{\beta}\left(  y,t\right)
-m_{4\beta}\right)  =\sum_{\alpha\neq\beta}f_{\alpha}\left(  y,t\right)
\left(  w_{\alpha}(x,t_{1})-w_{\alpha}(y,t)\right)  .
\]
Thus for all $\left(y,t\right)  \in Q(4R)$
\begin{align*}
\left(  W_{\beta}\left(  y,t\right)  -m_{4\beta}\right) &\leq\frac
{1}{\lambda_{\ast}}\sum_{\alpha\neq\beta}f_{\alpha}\left(  y,t\right)
\left(  w_{\alpha}(x,t_{1})-w_{\alpha}(y,t)\right)\\
& \leq\frac{1}
{\lambda_{\ast}}\sum_{\alpha\neq\beta}f_{\alpha}\left(  y,t\right) \left(
M_{4\alpha}-w_{\alpha}(y,t)\right) \\
&  \leq\frac{\Lambda_{\ast}}{\lambda_{\ast}}\sum_{\alpha\neq\beta}\left(
M_{4\alpha}-w_{\alpha}(y,t)\right).
\end{align*}
Hence, using convexity of $s \to s^p$ we have
\begin{align*}
\left(  W_{\beta}\left(  y,t\right)  -m_{4\beta}\right)  ^{p}  &  \leq\left[
\frac{\Lambda_{\ast}}{\lambda_{\ast}}\sum_{\alpha\neq\beta}\left(  M_{4\alpha
}-w_{\alpha}(y,t)\right)  \right]  ^{p}\\
&  \leq\left(  \frac{\Lambda_{\ast}}{\lambda_{\ast}}\right)  ^{p}N^{p-1}%
\sum_{\alpha\neq\beta}\left(  M_{4\alpha}-w_{\alpha}(y,t)\right)  ^{p}.%
\end{align*}
Integrating in the $(y,t)$ variables and applying (\ref{Lieb367a}) we have
\begin{align*}
\left(  R^{-n-2}\int_{\Theta(R)}\left(  W_{\beta}\left(  y,t\right)
-m_{4\beta}\right)  ^{p}\right)  ^{\frac{1}{p}}  &  \leq\frac{\Lambda_{\ast}%
}{\lambda_{\ast}}N^{1-\frac{1}{p}}\left(  R^{-n-2}\int_{\Theta(R)} \sum_{\alpha
\neq\beta}\left(  M_{4\alpha}-w_{\alpha}(y,t)\right)  ^{p}\right)
^{\frac{1}{p}}\\
&  \leq\frac{\Lambda_{\ast}}{\lambda_{\ast}}NC_{1}\left[  P(4R)-P(R)\right].
\end{align*}
Summing over $\beta$ we have
\begin{equation}
\sum_{\beta}\left(  R^{-n-2}\int_{\Theta(R)}\left(  W_{\beta}\left(
y,t\right)  -m_{4\beta}\right)  ^{p}\right)  ^{\frac{1}{p}}\leq\left(
N+1\right)  \frac{\Lambda_{\ast}}{\lambda_{\ast}}NC_{1}\left[
P(4R)-P(R)\right].  \label{bound1}%
\end{equation}
Also, a direct application of Theorem \ref{pweakharnack} yields
\begin{gather} \label{bound2}
\begin{split}
\sum_{\beta}\left(  R^{-n-2}\int_{\Theta(R)}\left(  M_{4\beta}-W_{\beta
}\left(  y,t\right)  \right)  ^{p}\right)  ^{\frac{1}{p}} &  \leq\sum_{\beta
}C_{1}\left(  M_{4\beta}-M_{1\beta}\right)\\
&  \leq C_{1}\left[  P(4R)-P(R)\right]
\end{split}
\end{gather}
Applying Minkowski's inequality, followed by (\ref{bound1}), (\ref{bound2})
gives
\begin{align*}
P(4R) &  =\sum_{\beta}M_{4\beta}-m_{4\beta}\\
&  =\sum_{\beta}\left(  R^{-n-2}\int_{\Theta(R)} \left( 
M_{4\beta}-m_{4\beta}\right)
^{p}\right)  ^{1/p}\\
& =\sum_{\beta}\left(  R^{-n-2}\int\left(  M_{4\beta
}-W_{\beta}\left(  y,t\right)  +W_{\beta}\left(  y,t\right)  -m_{4\beta
}\right)  ^{p}\right)  ^{1/p}\\
&  \leq \sum_{\beta} \left(  R^{-n-2}\int_{\Theta(R)}\left( 
M_{4\beta}-W_{\beta}\left(
y,t\right)  \right)  ^{p}\right)  ^{\frac{1}{p}}+\sum_{\beta}\left(
R^{-n-2}\int_{\Theta(R)}\left(  W_{\beta}\left(  y,t\right)  -m_{4\beta
}\right)  ^{p}\right)  ^{\frac{1}{p}}\\
&  \leq\left(  \left(  N+1\right)  \frac{\Lambda_{\ast}}{\lambda_{\ast}%
}N+1\right)  C_{1}\left[  P(4R)-P(R)\right].
\end{align*}
Rearranging this yields
\[
\left(  \left(  N+1\right)  \frac{\Lambda_{\ast}}{\lambda_{\ast}}N+1\right)
C_{1}P(R)\leq\left[  \left(  \left(  N+1\right)  \frac{\Lambda_{\ast}}%
{\lambda_{\ast}}N+1\right)  C_{1}-1\right]  P(4R),
\]
or in other words
\[
P(R)\leq\mu P(4R)
\]
for appropriately chosen $\mu$.  A standard iteration argument  (\cite[Theorem
8.23]{GT}) now yields%
\begin{equation}
P(\rho)\leq\left(  \frac{\rho}{R}\right)  ^{\alpha}P(R),\label{decaydetails1}%
\end{equation}
for%
\[
\alpha=\frac{1}{2}\frac{\log\frac{1}{\mu}}{\log(4)}.
\]
Now observe that in our application of (\cite{Lieb}, Lemma 14.5) we can
choose the set of vectors to contain all vectors of the form $e_{j} ,$
$\left(  e_{j} \pm e_{k}\right)  /\sqrt{2}$ and $\left(  e_{j} \pm \sqrt{-1}
e_{k}\right)  /\sqrt{2}$ for any particular coordinate basis. 
Trivially we observe%

\[
\sum_{j}\osc_{Q(\rho)}W_{j \bar{j}} \leq P(\rho),
\]
and then using polarization and the triangle inequality we may similarly bound
the oscillation of any of the components of $W.$

We also note that clearly
\begin{equation}
P(R)\leq N\osc_{Q(\rho)}W.\label{decaydetails3}%
\end{equation}

Thus 
\[
\osc_{Q(\rho)}W\leq C_2(N)
\left(  \frac{\rho}{R}\right)  ^{\alpha
}\osc_{Q(R)}W,
\]
which is the conclusion of the theorem.
\end{proof}

\begin{proof}[Proof of Theorem \ref{EvansKrylov}] Let us give the argument for
(\ref{complexparabolic}), the other cases being similar.  Let $u \in
\EE^{k,l}_{Q_2,\gl,\gL}$ be a solution to (\ref{complexparabolic}).  Observe
that uniform estimates on the complex Hessian of $u$ imply uniform upper and
lower bounds on the corresponding matrix $W$.  By (\ref{complexparequiv}) and
Proposition \ref{complexsubsoln}, we see that $u$ satisfies the hypotheses of
Theorem \ref{genEK}, and therefore we conclude a $C^{\ga}$ estimate for $\dt u$
and $W$.  Examining (\ref{cW}), we see that the lower right block  $-u^{\bw_j
w_i}$ has a
$C^{\ga}$ estimate, which implies that $u_{w_i \bw_j}$ does as well, using the
uniform lower bounds on the matrix.   Combining this estimate with the estimate
on the upper right and lower left blocks  we see that the derivatives $u_{z_i
\bw_j}$
and $u_{w_i \bz_j}$ also enjoy $C^{\ga}$ estimates.  Finally, these estimates
together with the estimates for the upper left block of $W$ imply a $C^{\ga}$
estimate for
$u_{z_i\bz_j}$.
\end{proof}

\begin{cor} \label{complexrigidity} Let $u_t$ be a solution to
(\ref{complexparabolic}) on $(-\infty,0] \times \mathbb C^n$ such that $u_t \in
\EE_{\mathbb C^n,\gl,\gL}^{k,l}$ for all $t \in (-\infty,0]$.  Then $\N \del
\delb u_t = \frac{\del^2 u}{\del t^2} = 0$
for all $t$.
\begin{proof} Suppose there exists a point such that $\brs{\N \del \delb u} \neq
0$.  By translating in space and time we can assume without loss of generality
this point is $(0,0)$.  Fix some $\mu > 0$ and consider
\begin{align*}
v(x,t) := \mu^{-2} u(\mu x, \mu^2 t) 
\end{align*}
By direct calculation one verifies that $v$ is a solution to
(\ref{complexparabolic}) on $(-\infty,0] \times \mathbb C^n$ and moreover $v_t
\in \EE_{\mathbb C^n,\gl,\gL}^{k,l}$ for all $t \in (-\infty,0]$.  Also observe
that $\brs{\N \del \delb v}(0,0) = \mu \brs{\N \del \delb u}(0,0)$.  However, we
know by Theorem \ref{EvansKrylov} that there is a uniform $C, \ga > 0$ so that
$\brs{\del\delb v}_{C^{\ga}(Q((0,0),1)} \leq C$.  By modifying $v$ by a
time-independent affine function we can ensure a uniform $C^0$ estimate for $v$
on $Q((0,0),1)$ as well.  At this point we may apply Schauder estimates on
$Q((0,0),1)$
to obtain an a priori bound for $\brs{\N \del \delb v}(0,0)$.  For $\mu$ chosen
sufficiently large this is a contradiction, finishing the proof.
\end{proof}
\end{cor}

\section{Evans-Krylov Regularity for Pluriclosed Flow} \label{PCFsec}
\subsection{Commuting Generalized K\"ahler manifolds} \label{compcfss}

In this section we exploit the estimates of \S \ref{eksec} to establish a priori
regularity results for the pluriclosed flow.  We briefly recall here the
discussion in (\cite{SPCFSTB}) wherein the pluriclosed flow in the setting of
generalized K\"ahler geometry with commuting complex structures is reduced to a
parabolic flow of the kind (\ref{complexparabolic}).  

Let $(M^{2n}, g, J_A, J_B)$ be a generalized K\"ahler manifold satisfying
$[J_A,J_B] = 0$.  Define
\begin{align*}
 \Pi := J_A J_B \in \End(TM).
\end{align*}
It follows that $\Pi^2 = \Id$, and $\Pi$ is $g$-orthogonal, hence $\Pi$ defines
a $g$-orthogonal decomposition into its $\pm 1$ eigenspaces, which we denote
\begin{align*}
 TM = T_+ M \oplus T_- M.
\end{align*}
Moreover, on the complex manifold $(M^{2n}, J_A)$ we can similarly decompose the
complexified tangent bundle $T_{\mathbb C}^{1,0}$.  For notational simplicity we
will denote
\begin{align*}
 T_{\pm}^{1,0} := \ker \left( \Pi \mp I \right) : T^{1,0}_{\mathbb C}
(M, J_A) \to T^{1,0}_{\mathbb C} (M, J_A).
\end{align*}
We use similar notation to denote the pieces of the complex cotangent bundle. 
Other tensor bundles inherit similar decompositions.  The one of most importance
to us is
\begin{align*}
 \Lambda^{1,1}_{\mathbb C}(M, J_A) =&\ \left(\Lambda^{1,0}_+ \oplus
\Lambda^{1,0}_- \right) \wedge \left(\Lambda^{0,1}_+ \oplus \Lambda^{0,1}_-
\right)\\
=&\ \left[\Lambda^{1,0}_+ \wedge \Lambda_+^{0,1} \right] \oplus
\left[\Lambda^{1,0}_+ \wedge \Lambda_-^{0,1} \right] \oplus
\left[\Lambda^{1,0}_- \wedge \Lambda_+^{0,1} \right] \oplus
\left[\Lambda^{1,0}_- \wedge \Lambda_-^{0,1} \right].
\end{align*}
Given $\mu \in \Lambda^{1,1}_{\mathbb C}(M, J_A)$ we will denote this
decomposition as
\begin{align} \label{oneoneproj}
 \mu := \mu^+ + \mu^{\pm} + \mu^{\mp} + \mu^-
\end{align}

\begin{defn} \label{chidef} Let $(M^{2n}, J_{A}, J_B)$ be a bicomplex manifold
such that $[J_A,J_B] = 0$.  Let
\begin{align*}
\chi(J_A,J_B) = c^+_1(T^{1,0}_+) - c^-_1(T^{1,0}_{+}) +
c_1^-(T^{1,0}_{-}) - c_1^+(T^{1,0}_{-}).
\end{align*}
The meaning of this formula is the following: fix Hermitian metrics $h_{\pm}$ on
the holomorphic line bundles $\det T^{1,0}_{\pm}$, and use these to define
elements of $c_1(T^{1,0}_{\pm})$, and then project according to the
decomposition (\ref{oneoneproj}).  In particular, given such metrics $h_{\pm}$
we let $\rho(h_{\pm})$ denote the associated representatives of
$c_1(T_{\pm}^{1,0})$, and then let
\begin{align*}
 \chi(h_{\pm}) = \rho^+(h_+) - \rho^{-}(h_+) + \rho^-(h_-) - \rho^+(h_-).
\end{align*}
This definition yields a well-defined class
in a certain cohomology group, defined in \cite{SPCFSTB}, which we now describe.
\end{defn}

\begin{defn} Let $(M^{2n}, J_A, J_B)$ be a bihermitian manifold with $[J_A,J_B]
= 0$.  Given $\phi_A \in \Lambda^{1,1}_{J_A,\mathbb R}$, let $\phi_B = - \phi_A(
\Pi
\cdot, \cdot) \in \Lambda^{1,1}_{J_B,\mathbb R}$.  We say that $\phi_A$ is
\emph{formally generalized K\"ahler} if
\begin{gather} \label{FGK}
\begin{split}
d^c_{J_A} \phi_A =&\ - d^c_{J_B} \phi_B\\
d d^c_{J_A} \phi_A =&\ 0.
\end{split}
\end{gather}
\end{defn}

\begin{defn} Let $(M^{2n}, g, J_A, J_B)$ denote a generalized K\"ahler manifold
such that $[J_A,J_B] = 0$.  Let
\begin{align*}
 \mathcal H := \frac{ \left\{ \phi_A \in \Lambda^{1,1}_{J_A, \mathbb R}\ |  \
\phi_A 
\mbox{ satisfies } (\ref{FGK}) \right\}}{ \left\{ \gd_+ \gd^c_+ f -
\gd_- \gd^c_- f \right\}}.
\end{align*}
\end{defn}

With this kind of cohomology space, we can define the analogous notion to the
``K\"ahler cone,'' which we refer to as $\mathcal P$, the ``positive cone.''

\begin{defn}
 Let $(M^{2n}, g, J_A, J_B)$ denote a generalized K\"ahler manifold
such that $[J_A,J_B] = 0$.  Let
\begin{align*}
 \mathcal P := \left\{ [\phi] \in \HH\ |\ \exists \gw \in [\phi], \gw > 0
\right\}.
\end{align*}
\end{defn}

\begin{defn} Let $(M^{2n}, g,J_{A}, J_B)$ be a generalized K\"ahler manifold
such that $[J_A,J_B] = 0$.  We say that $\chi = \chi(J_A,J_B) > 0$, (resp.
$(\chi < 0,\ \chi =
0)$ if $\chi \in \mathcal P$, (resp. $- \chi \in \mathcal P, \chi = 0$).
\end{defn}

\subsection{Scalar reduction and Evans-Krylov estimate} \label{pcfeksec}

In this subsection we describe how to reduce (\ref{PCF}) to a scalar PDE in the
setting of commuting generalized K\"ahler manifolds.  This is then used in
conjunction with Theorem \ref{EvansKrylov} to establish Corollary \ref{PCFEK}.

First we recall that it follows from (\cite{SPCFSTB} Proposition 3.2, Lemma 3.4)
that the pluriclosed flow in this setting reduces to
\begin{align} \label{PCFGK}
 \dt \gw =&\ - \chi(g_{\pm}).
\end{align}
From the discussion in \S \ref{compcfss} we see that a solution to (\ref{PCFGK})
induces a solution to an ODE in $\mathcal P$, namely
\begin{align*}
 [\gw_t] = [\gw_0] - t \chi.
\end{align*}
This suggests the following definition.
\begin{defn} \label{taustardef} Given $(M^{2n}, g, J_A, J_B)$ a generalized
K\"ahler manifold, let
\begin{align*}
 \tau^*(g) := \sup \left\{ t \geq 0 | [\gw] - t \chi \in \mathcal P \right\}.
\end{align*}
\end{defn}
Now fix $\tau < \tau^*$, so that by hypothesis if we fix arbitrary metrics
$\til{h}_{\pm}$ on $T^{1,0}_{\pm}$, there exists $a \in C^{\infty}(M)$ such that
\begin{align*}
 \gw_0 - \tau \chi(\til{h}_{\pm}) + \left( \gd_+ \gd^c_+ - \gd_- \gd^c_- \right)
a > 0.
\end{align*}
Now set $h_{\pm} = e^{\pm \frac{a}{2 \tau}} \til{h}_{\pm}$.  Thus $\gw_0 - \tau
\chi(h_{\pm}) > 0$, and by convexity it follows that
\begin{align*}
 \hat{\gw}_t := \gw_0 - t \chi(h_{\pm})
\end{align*}
is a smooth one-parameter family of metrics.  Furthermore, given a function $f
\in C^{\infty}(M)$, let
\begin{align*}
 \gw_f := \hat{\gw} + \left( \gd_+ \gd^c_+ - \gd_- \gd^c_- \right) f,
\end{align*}
with $g^f$ the associated Hermitian metric.  Now suppose that $u_t$ satisfies
\begin{align} \label{scalarPCF}
 \dt u =&\ \log \frac{\det g_+^u \det h_-}{\det h_+ \det g_-^u}.
\end{align}
An elementary calculation using the transgression formula for the first Chern
class (\cite{SPCFSTB} Lemma 3.4) yields that $\gw_u$ solves (\ref{PCFGK}).

\begin{proof}[Proof of Theorem \ref{PCFEK}] Since by hypothesis $\tau <
\tau^*$ we can adopt the discussion above and reduce our solution to the scalar
flow (\ref{scalarPCF}).  We recall the quantity $\gU = \gU(g,h)$ defined by the
difference of the Chern connections associated to $g$ and $h$.  In particular,
we have
\begin{align*}
\gU_{ij}^k =&\ \left( \N^g - \N^h \right)_{ij}^k = g^{\bl k} g_{i \bl,j} -
h^{\bl k} h_{i \bl,j}.
\end{align*}
For a reduced solution to pluriclosed flow as above, there exists some
background tensor $A = A(\hat{g},h,\tau)$ such that
\begin{align*}
\gU_{ij}^k =&\ g^{\bl k} \left[ (\gd_+ \gd_+^c - \gd_- \gd_-^c) u
\right]_{i\bl,j} + A_{ij}^k.
\end{align*}
By a standard argument using Schauder estimates, to prove the theorem
it suffices to show that there exists a uniform constant $C$ such that 
\begin{align*}
\sup_{M \times [0,\tau)} \brs{\gU(g,h)}^2 \leq C.
\end{align*}
If this were not the case, choose a sequence $(x_i,t_i)$ such that $t_i \to
\tau$
and 
\begin{align*}
\gl_i := \brs{\gU}^2(x_i,t_i) = \sup_{M \times [0,t_i]} \brs{\gU}^2.
\end{align*}
We now construct a blowup sequence of solutions centered at these points. 
Furthermore we can choose a small constant $R > 0$ and a
normal coordinate chart for $g_0$ centered at $x_i$, covering $B_R(x_i,
g_0)$.  Using such charts for each $i$ and translating in time we obtain a
solution to
(\ref{scalarPCF}) on $Q((0,0),\rho)$ for some small uniform constant $\rho > 0$
such that $\brs{\gU_i}^2(0,0) = \gl_i$.

Now fix some constant $B > 0$, let $\ga_i := B \gl_i$ and let
\begin{align*}
 \til{u}_i =&\ \ga_i^2 \left( u_i(\ga_i^{-1} x, \ga_i^{-2} t) - u_i(0,0)
\right)\\
 \til{\gw}_i =&\ \hat{\gw} \left(\ga_i^{-1} x, \ga_i^{-2} t \right)\\
 \til{h}_i =&\ h \left(\ga_i^{-1} x, \ga_i^{-2} t \right).
 \end{align*}
A direct calculation shows that for each $i$ the function $\til{u}_i$ is a
solution of
\begin{align*}
 \dt \til{u}_i = \log \frac{ \det \til{g}_+^{\til{u}_i} \det \til{h}-}{\det
\til{h}_+
\det \til{g}_+^{\til{u}_i}}.
\end{align*}
For sufficiently large $i$ we may assume this solution exists on $Q((0,0),3)$. 
Moreover, we claim that for every $j \in \mathbb N$ we have
$\nm{\til{u}_i}{C^j(Q((0,0),3))} \leq
C(j,B)$.  By the choice of scaling parameters we have that
$\brs{\til{\gU}}_{\til{g}}^2 \leq C B$.  This corresponds to a uniform $C^3$
estimate for $\til{u}_i$, which after the application of Schauder estimates
implies uniform $C^j$ bounds for all $j$.  Given this, using standard
compactness
theorems we obtain
a subsequence of $(\til{u}_i,\til{\gw}_i,\til{h}_i)$ converging to a limit
$(u_{\infty},\gw_{\infty},h_{\infty})$ on $Q((0,0),2)$.  As the background
metrics $\gw$, $h$ were uniformly controlled before the blowup it follows that
$\gw_{\infty}, h_{\infty}$ are flat, and moreover uniformly equivalent to the
standard flat metric with a bound depending only on the background data $\gw,
h$.  By a rescaling in space and time and adding a function of time only, we can
assume that $u_{\infty}$ is a solution to (\ref{complexparabolic}), which
moreover by construction satisfies
\begin{align*}
 u_{\infty} \in&\ \EE^{k,l}_{\gl,\gL,Q((0,0),2)}, \qquad \brs{\N^3
u_{\infty}}(0,0) = c(\gl,\gL,\gw,h) B^{-1}.
\end{align*}
However, as a solution to (\ref{complexparabolic}), by Theorem \ref{EvansKrylov}
and Schauder estimates there is an a priori interior $C^3$ estimate for
$u_{\infty}$.  For $B$ chosen sufficiently small this is a contradiction,
finishing the proof.
\end{proof}

\section{Appendix: Evolution Equations} \label{calcsec}

In this section we prove the crucial subsolution properties for the matrix $W$
along the real and complex twisted Monge-Amp\`ere equations.  The results are
contained in Lemmas \ref{Wev} and \ref{complexevs}.  We directly prove the case
of complex variables first, which consists of lengthy
calculations and applications of the Cauchy-Schwarz inequality.  Again we note
that these
monotonicity properties are suggested by the discussion of Legendre
transformations in \S \ref{legbck}.  A similar direct calculation can yield the
case of real variables, but we suppress this as it is lengthy and nearly
identical to the complex case.  Instead we show that the real case follows by
formally extending variables and appealing to the complex setting.  

\begin{lemma} \label{complexevs} Let $u_t$ be a solution to
(\ref{complexparabolic}) such that $u_t
\in \mathcal E^{k,l}_U$ for all $t$.  Then
\begin{align} \label{cpmp3}
\left( \dt - \LL \right) \frac{\del u}{\del t} =&\ 0.
\end{align}
Also,
\begin{align*}
\left(\dt - \LL \right)W = Q,
\end{align*}
where
\begin{gather}
\begin{split} \label{cpmp1}
Q_{\ga_z \bga_z} =&\ - u^{\bz_q z_r} u^{\bz_s z_p} u_{z_p
\bz_q \ga_z} u_{z_r \bz_s \bga_z} + u^{\bz_q z_r} u^{\bz_s z_p}
u_{z_p \bz_q \ga_z} u_{z_r \bz_s \bw_k} u^{\bw_k w_l} u_{w_l
\bga_z}\\
&\ - u_{\ga_z \bw_k} u^{\bw_k w_p} u^{\bz_q z_r} u^{\bz_s z_p} u_{z_p \bz_q w_p}
u_{z_r \bz_s \bw_q}
u^{\bw_q w_l} u_{w_l \bga_z} + u_{\ga_z \bw_k} u^{\bw_k w_l} 
u^{\bz_q z_r} u^{\bz_s z_p} u_{z_p \bz_q w_l} u_{z_r \bz_s \bga_z}\\
&\ + u^{\bz_b z_a} \left[ -
u_{\ga_z \bw_k z_a} u^{\bw_k w_p} u_{w_p \bw_q \bz_b} u^{\bw_q w_l} u_{w_l
\bga_z}
+ u_{\ga_z \bw_k z_a} u^{\bw_k w_l} u_{w_l \bga_z \bz_b} \right.\\
&\ - u_{\ga_z \bw_k \bz_b} u^{\bw_k w_p} u_{w_p \bw_q z_a} u^{\bw_q w_l} u_{w_l
\bga_z} + u_{\ga_z \bw_k} u^{\bw_k w_r} u_{w_r \bw_s \bz_b} u^{\bw_s w_p} u_{w_p
\bw_q z_a} u^{\bw_q w_l} u_{w_l \bga_z}\\
&\ + u_{\ga_z \bw_k} u^{\bw_k w_p} u_{w_p \bw_q z_a} u^{\bw_q w_r} u_{w_r
\bw_s \bz_b} u^{\bw_s w_l} u_{w_l \bga_z} - u_{\ga_z \bw_k} u^{\bw_k w_p} u_{w_p
\bw_q z_a} u^{\bw_q w_l} u_{w_l \bga_z
\bz_b}\\
&\ \left. + u_{\ga_z \bw_k \bz_b} u^{\bw_k w_l} u_{w_l \bga_z z_a} - u_{\ga_z
\bw_k}
u^{\bw_k w_r} u_{w_r \bw_s \bz_b} u^{\bw_s w_l} u_{w_l \bga_z z_a} \right]\\
&\ - u^{\bw_b w_a} \left[ - u_{\ga_z \bw_k \bw_b} u^{\bw_k w_p} u_{w_p \bw_q
w_a}
u^{\bw_q w_l} u_{w_l
\bga_z} + u_{\ga_z \bw_k} u^{\bw_k w_r} u_{w_r \bw_s \bw_b} u^{\bw_s w_p} u_{w_p
\bw_q w_a} u^{\bw_q w_l} u_{w_l \bga_z} \right.\\
&\ \left. + u_{\ga_z \bw_k \bw_b} u^{\bw_k w_l} u_{w_l \bga_z w_a} - u_{\ga_z
\bw_k}
u^{\bw_k w_r} u_{w_r \bw_s \bw_b} u^{\bw_s w_l} u_{w_l \bga_z w_a} \right]
\end{split}
\end{gather}
\begin{gather}
\begin{split} \label{cpmp2}
Q_{\ga_w \bga_w} =&\ - u^{\bga_w w_k} u^{\bw_l \ga_w} u^{\bz_q z_r} u^{\bz_s
z_p} u_{z_p \bz_q w_k}
u_{z_r \bz_s \bw_l} + u^{\bz_l z_k} u^{\bga_w w_p} u_{w_p \bw_q \bz_l} u^{\bw_q
w_j} u_{w_j
\bw_k z_k} u^{\bw_k \ga_w}\\
&\ + u^{\bz_l z_k} u^{\bga_w w_j} u_{w_j \bw_k z_k} u^{\bw_k w_p} u_{w_p
\bw_q \bz_l} u^{\bw_q \ga_w}  - u^{\bw_l w_k} u^{\bga_w w_p} u_{w_p \bw_q \bw_l}
u^{\bw_q w_j} u_{w_j
\bw_r w_k} u^{\bw_r \ga_w},
\end{split}
\end{gather}
\begin{gather}
\begin{split} \label{cpmp4}
Q_{\ga_z \ga_w} =&\ - u^{\bz_b z_a} u^{\bz_d z_c} u_{z_a \bz_d \ga_z} u_{z_c
\bz_b \bw_k} u^{\bw_k \ga_w} + u_{\ga_z \bw_k} u^{\bw_k w_l} u^{\bw_p \ga_w} 
u^{\bz_b z_a} u^{\bz_d z_c} u_{z_a \bz_d w_l} u_{z_c \bz_b \bw_p}\\
&\ - u^{\bz_b z_a} \left[  - u_{\ga_z \bw_k z_a} u^{\bw_k w_p} u_{w_p \bw_q
\bz_b} u^{\bw_q \ga_w}  - u_{\ga_z \bw_k \bz_b} u^{\bw_k w_p} u_{w_p \bw_q z_a}
u^{\bw_q \ga_w} \right.\\
&\ \left. + u_{\ga_z \bw_k} u^{\bw_k w_r} u_{w_r \bw_s \bz_b} u^{\bw_s w_p}
u_{w_p \bw_q z_a} u^{\bw_q \ga_w} + u_{\ga_z \bw_k} u^{\bw_k w_p} u_{w_p \bw_q
z_a} u^{\bw_q w_r} u_{w_r \bw_s \bz_b} u^{\bw_s \ga_w} \right]\\
&\ + u^{\bw_b w_a} \left[  - u_{\ga_z \bw_k \bw_b} u^{\bw_k w_p} u_{w_p \bw_q
w_a} u^{\bw_q \ga_w} + u_{\ga_z \bw_k} u^{\bw_k w_r} u_{w_r \bw_s \bw_b}
u^{\bw_s w_p} u_{w_p \bw_q w_a} u^{\bw_q \ga_w}\right],
\end{split}
\end{gather}
\begin{gather}
\begin{split} \label{cpmp5}
Q_{\bga_z \bga_w} =\bar{Q}_{\ga_z \ga_w}
\end{split}
\end{gather}
\begin{proof} First we prove (\ref{cpmp3}).
\begin{align*}
 \dt \left( \frac{\del u}{\del t} \right) =&\ \dt \left( \log \det u_{\ga_z
\bga_z}
- \log
\det (-u_{\ga_w \bga_w}) \right)\\
 =&\ u^{\bz_b z_a} \left( \frac{\del u}{\del t} \right)_{z_a \bz_b} - u^{\bw_b
w_a}
\left( \frac{\del u}{\del t} \right)_{w_a \bw_b}\\
 =&\ \LL \frac{\del u}{\del t}.
\end{align*}
Next we establish (\ref{cpmp2}).  We start by computing partial derivatives
\begin{gather} \label{cpmp210}
\begin{split}
\left(\log \det u_{z \bz} \right)_{,\ga\gb} =&\ \left( u^{\bz_q z_p} u_{z_p
\bz_q \ga} \right)_{,\gb}= u^{\bz_q z_p} u_{z_p \bz_q \ga \gb} - u^{\bz_q z_r}
u^{\bz_s z_p} u_{z_p \bz_q \ga} u_{z_r \bz_s \gb},\\
\left(\log \det (- u_{yy}) \right)_{,\ga\gb} =&\ \left( u^{\bw_q w_p} u_{w_p
\bw_q \ga} \right)_{,\gb}= u^{\bw_q w_p} u_{w_p \bw_q \ga \gb} - u^{\bw_q w_r}
u^{\bw_s w_p} u_{w_p \bw_q \ga} u_{w_r \bw_s \gb},\\
\end{split}
\end{gather}
Using this we compute
\begin{gather} \label{cpmp115}
\begin{split}
\dt u^{\bga_w \ga_w} =&\ - u^{\bga_w w_k} \left( \dt u \right)_{w_k \bw_l}
u^{\bw_l
\ga_w}\\
=&\ - u^{\bga_w w_k} \left( \log \det u_{z \bz} - \log \det (- u_{w \bw})
\right)_{w_k \bw_l} u^{\bw_l \ga_w}\\
=&\ u^{\bga_w w_k} u^{\bw_l \ga_w} \left( u^{\bw_q w_p} u_{w_p \bw_q w_k \bw_l}
-
u^{\bw_q w_r} u^{\bw_s w_p} u_{w_p \bw_q w_k} u_{w_r \bw_s \bw_l}\right.\\
&\ \left. \qquad - u^{\bz_q z_p} u_{z_p \bz_q w_k \bw_l} + u^{\bz_q z_r}
u^{\bz_s z_p} u_{z_p \bz_q w_k} u_{z_r \bz_s \bw_l} \right).
\end{split}
\end{gather}
Also we compute the partial derivatives
\begin{gather} \label{cpmp220}
\begin{split}
u^{\bga_w \ga_w}_{\ga\gb} =&\ - \left( u^{\bga_w w_j} u_{w_j \bw_k \ga} u^{\bw_k
\ga_w} \right)_{,\gb}\\
=&\ u^{\bga_w w_p} u_{w_p \bw_q \gb} u^{\bw_q w_j} u_{w_j \bw_k \ga} u^{\bw_k
\ga_w} - u^{\bga_w w_j} u_{w_j \bw_k \ga \gb} u^{\bw_k \ga_w}\\
&\ + u^{\bga_w w_j} u_{w_j \bw_k \ga} u^{\bw_k w_p} u_{w_p \bw_q \gb} u^{\bw_q
\ga_w}.
\end{split}
\end{gather}
Thus we have
\begin{gather} \label{cpmp230}
\begin{split}
\LL & (u^{\bga_w \ga_w})\\
=&\ u^{\bz_l z_k} (u^{\bga_w \ga_w})_{,z_k \bz_l} - u^{\bw_k w_l} (u^{\bga_w
\ga_w})_{,w_k \bw_l}\\
=&\ u^{\bz_l z_k} \left( u^{\bga_w w_p} u_{w_p \bw_q \bz_l} u^{\bw_q w_j} u_{w_j
\bw_k z_k} u^{\bw_k \ga_w} - \right.\\
&\ \left. \qquad u^{\bga_w w_j} u_{w_j \bw_k z_k \bz_l} u^{\bw_k
\ga_w} +
u^{\bga_w w_j} u_{w_j \bw_k z_k} u^{\bw_k w_p} u_{w_p \bw_q \bz_l} u^{\bw_q
\ga_w}
\right)\\
&\ - u^{\bw_l w_k} \left(u^{\bga_w w_p} u_{w_p \bw_q \bw_l} u^{\bw_q w_j} u_{w_j
\bw_r w_k} u^{\bw_r \ga_w} -\right.\\
&\ \left. \qquad u^{\bga_w w_j} u_{w_j \bw_r w_k \bw_l} u^{\bw_r
\ga_w} +
u^{\bga_w w_j} u_{w_j \bw_r w_k} u^{\bw_r w_p} u_{w_p \bw_q \bw_l} u^{\bw_q
\ga_w}
\right).
\end{split}
\end{gather}
Putting together (\ref{cpmp115}) and (\ref{cpmp230}) yields
\begin{align*}
\left( \dt - \LL \right) & W_{\ga_w \bga_w}\\
=&\ - u^{\bga_w w_k} u^{\bw_l \ga_w} \left( - u^{\bw_q w_r} u^{\bw_s w_p} u_{w_p
\bw_q w_k} u_{w_r \bw_s \bw_l} + u^{\bz_q z_r} u^{\bz_s z_p} u_{z_p \bz_q w_k}
u_{z_r \bz_s \bw_l} \right)\\
&\ + u^{\bz_l z_k} \left( u^{\bga_w w_p} u_{w_p \bw_q \bz_l} u^{\bw_q w_j}
u_{w_j
\bw_k z_k} u^{\bw_k \ga_w} + u^{\bga_w w_j} u_{w_j \bw_k z_k} u^{\bw_k w_p}
u_{w_p
\bw_q \bz_l} u^{\bw_q \ga_w} \right)\\
&\ - u^{\bw_l w_k} \left(u^{\bga_w w_p} u_{w_p \bw_q \bw_l} u^{\bw_q w_j} u_{w_j
\bw_r w_k} u^{\bw_r \ga_w} + u^{\bga_w w_j} u_{w_j \bw_r w_k} u^{\bw_r w_p}
u_{w_p
\bw_q \bw_l} u^{\bw_q \ga_w} \right)\\
=&\ - u^{\bga_w w_k} u^{\bw_l \ga_w} u^{\bz_q z_r} u^{\bz_s z_p} u_{z_p \bz_q
w_k}
u_{z_r \bz_s \bw_l} + u^{\bz_l z_k} u^{\bga_w w_p} u_{w_p \bw_q \bz_l} u^{\bw_q
w_j} u_{w_j
\bw_k z_k} u^{\bw_k \ga_w}\\
&\ + u^{\bz_l z_k} u^{\bga_w w_j} u_{w_j \bw_k z_k} u^{\bw_k w_p} u_{w_p
\bw_q \bz_l} u^{\bw_q \ga_w}  - u^{\bw_l w_k} u^{\bga_w w_p} u_{w_p \bw_q \bw_l}
u^{\bw_q w_j} u_{w_j
\bw_r w_k} u^{\bw_r \ga_w},
\end{align*}
finishing the proof of (\ref{cpmp2}).  Next we establish
(\ref{cpmp1}).  First we compute using (\ref{cpmp210})
\begin{align*}
\dt u_{\ga_z \bga_z} =&\ \left( \log \det u_{\ga_z \bga_z} - \log \det (-
u_{\ga_w
\bga_w}) \right)_{\ga_z \bga_z}\\
=&\ u^{\bz_q z_p} u_{z_p \bz_q \ga_z \bga_z} - u^{\bz_q z_r} u^{\bz_s z_p}
u_{z_p
\bz_q \ga_z} u_{z_r \bz_s \bga_z} - u^{\bw_q w_p} u_{w_p \bw_q \ga_z \bga_z} +
u^{\bw_q w_r} u^{\bw_s w_p} u_{w_p \bw_q \ga_z} u_{w_r \bw_s \bga_z}.
\end{align*}
Also
\begin{align*}
\LL u_{\ga_z \bga_z} =&\ u^{\bz_q z_p} u_{\ga_z \bga_z z_p \bz_q} - u^{\bw_q
w_p}
u_{\ga_z \bga_z w_p \bw_q}.
\end{align*}
Thus
\begin{gather} \label{cpmp105}
\begin{split}
\left( \dt - \LL \right) u_{\ga_z \bga_z} =&\ - u^{\bz_q z_r} u^{\bz_s z_p}
u_{z_p
\bz_q \ga_z} u_{z_r \bz_s \bga_z} + u^{\bw_q w_r} u^{\bw_s w_p} u_{w_p \bw_q
\ga_z}
u_{w_r \bw_s \bga_z}.
\end{split}
\end{gather}
To compute the next term we first differentiate using (\ref{cpmp210})
\begin{gather} \label{cpmp110}
\begin{split}
\dt & \left( u_{\ga_z \bw_k} u^{\bw_k w_l} u_{w_l \bga_z} \right)\\
=&\ \left( \dt u \right)_{\ga_z \bw_k} u^{\bw_k w_l} u_{w_l \bga_z} - u_{\ga_z
\bw_k}
u^{\bw_k w_p} \left( \dt u \right)_{w_p \bw_q} u^{\bw_q w_l} u_{w_l \bga_z} +
u_{\ga_z \bw_k} u^{\bw_k w_l} \left( \dt u_{w_l \bga_z} \right)\\
=&\ \left( u^{\bz_q z_p} u_{z_p \bz_q \ga_z \bw_k} - u^{\bz_q z_r} u^{\bz_s z_p}
u_{z_p \bz_q \ga_z} u_{z_r \bz_s \bw_k} - u^{\bw_b
w_a} u_{w_a \bw_b \ga_z \bw_k} + u^{\bw_d w_c}
u^{\bw_b w_a} u_{w_a \bw_d \ga_z} u_{w_c \bw_b \bw_k} \right) u^{\bw_k w_l}
u_{w_l
\bga_z}\\
&\ - u_{\ga_z \bw_k} u^{\bw_k w_p} \left( u^{\bz_q z_p} u_{z_p \bz_q w_p \bw_q}
-
u^{\bz_q z_r} u^{\bz_s z_p} u_{z_p \bz_q w_p} u_{z_r \bz_s \bw_q} \right.\\
&\ \left. \qquad \qquad \qquad - u^{\bw_b
w_a} u_{w_a \bw_b w_p \bw_q} + u^{\bw_d w_c}
u^{\bw_b w_a} u_{w_a \bw_d w_p} u_{w_c \bw_b \bw_q} \right)
u^{\bw_q w_l} u_{w_l \bga_z}\\
&\ + u_{\ga_z \bw_k} u^{\bw_k w_l} \left( u^{\bz_q z_p} u_{z_p \bz_q w_l \bga_z}
-
u^{\bz_q z_r} u^{\bz_s z_p} u_{z_p \bz_q w_l} u_{z_r \bz_s \bga_z} - u^{\bw_b
w_a} u_{w_a \bw_b w_l \bga_z} + u^{\bw_d w_c}
u^{\bw_b w_a} u_{w_a \bw_d w_l} u_{w_c \bw_b \bga_z} \right).
\end{split}
\end{gather}
Next we compute the partial derivatives
\begin{align*}
\left(u_{\ga_z \bw_k} u^{\bw_k w_l} u_{w_l \bga_z}\right)_{,\ga\gb} =&\ \left(
u_{\ga_z \bw_k \ga} u^{\bw_k w_l} u_{w_l \bga_z} - u_{\ga_z \bw_k} u^{\bw_k w_p}
u_{w_p \bw_q \ga} u^{\bw_q w_l} u_{w_l \bga_z} + u_{\ga_z \bw_k} u^{\bw_k w_l}
u_{w_l \bga_z \ga} \right)_{,\gb}\\
=&\ u_{\ga_z \bw_k \ga \gb} u^{\bw_k w_l} u_{w_l \bga_z} - u_{\ga_z \bw_k \ga}
u^{\bw_k w_p} u_{w_p \bw_q \gb} u^{\bw_q w_l} u_{w_l \bga_z} + u_{\ga_z \bw_k
\ga}
u^{\bw_k w_l} u_{w_l \bga_z \gb}\\
&\ - u_{\ga_z \bw_k \gb} u^{\bw_k w_p} u_{w_p \bw_q \ga} u^{\bw_q w_l} u_{w_l
\bga_z} + u_{\ga_z \bw_k} u^{\bw_k w_r} u_{w_r \bw_s \gb} u^{\bw_s w_p} u_{w_p
\bw_q \ga} u^{\bw_q w_l} u_{w_l \bga_z}\\
&\ - u_{\ga_z \bw_k} u^{\bw_k w_p} u_{w_p \bw_q \ga \gb} u^{\bw_q w_l} u_{w_l
\bga_z} + u_{\ga_z \bw_k} u^{\bw_k w_p} u_{w_p \bw_q \ga} u^{\bw_q w_r} u_{w_r
\bw_s \gb} u^{\bw_s w_l} u_{w_l \bga_z}\\
&\ - u_{\ga_z \bw_k} u^{\bw_k w_p} u_{w_p \bw_q \ga} u^{\bw_q w_l} u_{w_l \bga_z
\gb}\\
&\ + u_{\ga_z \bw_k \gb} u^{\bw_k w_l} u_{w_l \bga_z \ga} - u_{\ga_z \bw_k}
u^{\bw_k
w_r} u_{w_r \bw_s \gb} u^{\bw_s w_l} u_{w_l \bga_z \ga} + u_{\ga_z \bw_k}
u^{\bw_k
w_l} u_{w_l \bga_z \ga \gb}.
\end{align*}
Thus we have
\begin{gather} \label{cpmp120}
\begin{split}
\LL & ( u_{\ga_z \bw_k} u^{\bw_k w_l} u_{w_l \bga_z})\\
=&\ u^{\bz_b z_a} \left(u_{\ga_z \bw_k} u^{\bw_k w_l} u_{w_l
\bga_z}\right)_{,z_a
\bz_b} - u^{\bw_b w_a} \left(u_{\ga_z \bw_k} u^{\bw_k w_l} u_{w_l
\bga_z}\right)_{,w_a \bw_b}\\
=&\ u^{\bz_b z_a} \left[ u_{\ga_z \bw_k z_a \bz_b} u^{\bw_k w_l} u_{w_l \bga_z}
-
u_{\ga_z \bw_k z_a} u^{\bw_k w_p} u_{w_p \bw_q \bz_b} u^{\bw_q w_l} u_{w_l
\bga_z}
+ u_{\ga_z \bw_k z_a} u^{\bw_k w_l} u_{w_l \bga_z \bz_b} \right.\\
&\ - u_{\ga_z \bw_k \bz_b} u^{\bw_k w_p} u_{w_p \bw_q z_a} u^{\bw_q w_l} u_{w_l
\bga_z} + u_{\ga_z \bw_k} u^{\bw_k w_r} u_{w_r \bw_s \bz_b} u^{\bw_s w_p} u_{w_p
\bw_q z_a} u^{\bw_q w_l} u_{w_l \bga_z}\\
&\ - u_{\ga_z \bw_k} u^{\bw_k w_p} u_{w_p \bw_q z_a \bz_b} u^{\bw_q w_l} u_{w_l
\bga_z} + u_{\ga_z \bw_k} u^{\bw_k w_p} u_{w_p \bw_q z_a} u^{\bw_q w_r} u_{w_r
\bw_s \bz_b} u^{\bw_s w_l} u_{w_l \bga_z}\\
&\ - u_{\ga_z \bw_k} u^{\bw_k w_p} u_{w_p \bw_q z_a} u^{\bw_q w_l} u_{w_l \bga_z
\bz_b}\\
&\ \left. + u_{\ga_z \bw_k \bz_b} u^{\bw_k w_l} u_{w_l \bga_z z_a} - u_{\ga_z
\bw_k}
u^{\bw_k w_r} u_{w_r \bw_s \bz_b} u^{\bw_s w_l} u_{w_l \bga_z z_a} + u_{\ga_z
\bw_k} u^{\bw_k w_l} u_{w_l \bga_z z_a \bz_b} \right]\\
&\ - u^{\bw_b w_a} \left[ u_{\ga_z \bw_k w_a \bw_b} u^{\bw_k w_l} u_{w_l \bga_z}
-
u_{\ga_z \bw_k w_a} u^{\bw_k w_p} u_{w_p \bw_q \bw_b} u^{\bw_q w_l} u_{w_l
\bga_z}
+ u_{\ga_z \bw_k w_a} u^{\bw_k w_l} u_{w_l \bga_z \bw_b} \right.\\
&\ - u_{\ga_z \bw_k \bw_b} u^{\bw_k w_p} u_{w_p \bw_q w_a} u^{\bw_q w_l} u_{w_l
\bga_z} + u_{\ga_z \bw_k} u^{\bw_k w_r} u_{w_r \bw_s \bw_b} u^{\bw_s w_p} u_{w_p
\bw_q w_a} u^{\bw_q w_l} u_{w_l \bga_z}\\
&\ - u_{\ga_z \bw_k} u^{\bw_k w_p} u_{w_p \bw_q w_a \bw_b} u^{\bw_q w_l} u_{w_l
\bga_z} + u_{\ga_z \bw_k} u^{\bw_k w_p} u_{w_p \bw_q w_a} u^{\bw_q w_r} u_{w_r
\bw_s \bw_b} u^{\bw_s w_l} u_{w_l \bga_z}\\
&\ - u_{\ga_z \bw_k} u^{\bw_k w_p} u_{w_p \bw_q w_a} u^{\bw_q w_l} u_{w_l \bga_z
\bw_b}\\
&\ \left. + u_{\ga_z \bw_k \bw_b} u^{\bw_k w_l} u_{w_l \bga_z w_a} - u_{\ga_z
\bw_k}
u^{\bw_k w_r} u_{w_r \bw_s \bw_b} u^{\bw_s w_l} u_{w_l \bga_z w_a} + u_{\ga_z
\bw_k} u^{\bw_k w_l} u_{w_l \bga_z w_a \bw_b} \right].
\end{split}
\end{gather}
Putting together (\ref{cpmp105}), (\ref{cpmp110}) and (\ref{cpmp120}) yields
\begin{align*}
 \left( \dt - \LL \right) & W_{\ga_z \bga_z}\\
 =&\ - u^{\bz_q z_r} u^{\bz_s z_p} u_{z_p
\bz_q \ga_z} u_{z_r \bz_s \bga_z} + u^{\bw_q w_r} u^{\bw_s w_p} u_{w_p \bw_q
\ga_z}
u_{w_r \bw_s \bga_z}\\
&\ - \left(  - u^{\bz_q z_r} u^{\bz_s z_p}
u_{z_p \bz_q \ga_z} u_{z_r \bz_s \bw_k} + u^{\bw_d w_c}
u^{\bw_b w_a} u_{w_a \bw_d \ga_z} u_{w_c \bw_b \bw_k} \right) u^{\bw_k w_l}
u_{w_l
\bga_z}\\
&\ + u_{\ga_z \bw_k} u^{\bw_k w_p} \left( - u^{\bz_q z_r} u^{\bz_s z_p} u_{z_p
\bz_q w_p} u_{z_r \bz_s \bw_q} + u^{\bw_d w_c}
u^{\bw_b w_a} u_{w_a \bw_d w_p} u_{w_c \bw_b \bw_q} \right)
u^{\bw_q w_l} u_{w_l \bga_z}\\
&\ - u_{\ga_z \bw_k} u^{\bw_k w_l} \left( -
u^{\bz_q z_r} u^{\bz_s z_p} u_{z_p \bz_q w_l} u_{z_r \bz_s \bga_z} + u^{\bw_d
w_c}
u^{\bw_b w_a} u_{w_a \bw_d w_l} u_{w_c \bw_b \bga_z} \right)\\
&\ + u^{\bz_b z_a} \left[ -
u_{\ga_z \bw_k z_a} u^{\bw_k w_p} u_{w_p \bw_q \bz_b} u^{\bw_q w_l} u_{w_l
\bga_z}
+ u_{\ga_z \bw_k z_a} u^{\bw_k w_l} u_{w_l \bga_z \bz_b} \right.\\
&\ - u_{\ga_z \bw_k \bz_b} u^{\bw_k w_p} u_{w_p \bw_q z_a} u^{\bw_q w_l} u_{w_l
\bga_z} + u_{\ga_z \bw_k} u^{\bw_k w_r} u_{w_r \bw_s \bz_b} u^{\bw_s w_p} u_{w_p
\bw_q z_a} u^{\bw_q w_l} u_{w_l \bga_z}\\
&\ + u_{\ga_z \bw_k} u^{\bw_k w_p} u_{w_p \bw_q z_a} u^{\bw_q w_r} u_{w_r
\bw_s \bz_b} u^{\bw_s w_l} u_{w_l \bga_z} - u_{\ga_z \bw_k} u^{\bw_k w_p} u_{w_p
\bw_q z_a} u^{\bw_q w_l} u_{w_l \bga_z
\bz_b}\\
&\ \left. + u_{\ga_z \bw_k \bz_b} u^{\bw_k w_l} u_{w_l \bga_z z_a} - u_{\ga_z
\bw_k}
u^{\bw_k w_r} u_{w_r \bw_s \bz_b} u^{\bw_s w_l} u_{w_l \bga_z z_a} \right]\\
&\ - u^{\bw_b w_a} \left[ -
u_{\ga_z \bw_k w_a} u^{\bw_k w_p} u_{w_p \bw_q \bw_b} u^{\bw_q w_l} u_{w_l
\bga_z}
+ u_{\ga_z \bw_k w_a} u^{\bw_k w_l} u_{w_l \bga_z \bw_b} \right.\\
&\ - u_{\ga_z \bw_k \bw_b} u^{\bw_k w_p} u_{w_p \bw_q w_a} u^{\bw_q w_l} u_{w_l
\bga_z} + u_{\ga_z \bw_k} u^{\bw_k w_r} u_{w_r \bw_s \bw_b} u^{\bw_s w_p} u_{w_p
\bw_q w_a} u^{\bw_q w_l} u_{w_l \bga_z}\\
&\ + u_{\ga_z \bw_k} u^{\bw_k w_p} u_{w_p \bw_q w_a} u^{\bw_q w_r} u_{w_r
\bw_s \bw_b} u^{\bw_s w_l} u_{w_l \bga_z} - u_{\ga_z \bw_k} u^{\bw_k w_p} u_{w_p
\bw_q w_a} u^{\bw_q w_l} u_{w_l \bga_z
\bw_b}\\
&\ \left. + u_{\ga_z \bw_k \bw_b} u^{\bw_k w_l} u_{w_l \bga_z w_a} - u_{\ga_z
\bw_k}
u^{\bw_k w_r} u_{w_r \bw_s \bw_b} u^{\bw_s w_l} u_{w_l \bga_z w_a} \right]\\
=&:\ \sum_{i=1}^{24} A_i.
\end{align*}
We observe that $A_2 + A_{18} = A_4 + A_{17} = A_6 + A_{21} = A_8 + A_{22} = 0$,
and hence
\begin{align*}
 \left( \dt - \LL \right) & W_{\ga_z \bga_z}\\
 =&\ - u^{\bz_q z_r} u^{\bz_s z_p} u_{z_p
\bz_q \ga_z} u_{z_r \bz_s \bga_z} + u^{\bz_q z_r} u^{\bz_s z_p}
u_{z_p \bz_q \ga_z} u_{z_r \bz_s \bw_k} u^{\bw_k w_l} u_{w_l
\bga_z}\\
&\ - u_{\ga_z \bw_k} u^{\bw_k w_p} u^{\bz_q z_r} u^{\bz_s z_p} u_{z_p \bz_q w_p}
u_{z_r \bz_s \bw_q}
u^{\bw_q w_l} u_{w_l \bga_z} + u_{\ga_z \bw_k} u^{\bw_k w_l} 
u^{\bz_q z_r} u^{\bz_s z_p} u_{z_p \bz_q w_l} u_{z_r \bz_s \bga_z}\\
&\ + u^{\bz_b z_a} \left[ -
u_{\ga_z \bw_k z_a} u^{\bw_k w_p} u_{w_p \bw_q \bz_b} u^{\bw_q w_l} u_{w_l
\bga_z}
+ u_{\ga_z \bw_k z_a} u^{\bw_k w_l} u_{w_l \bga_z \bz_b} \right.\\
&\ - u_{\ga_z \bw_k \bz_b} u^{\bw_k w_p} u_{w_p \bw_q z_a} u^{\bw_q w_l} u_{w_l
\bga_z} + u_{\ga_z \bw_k} u^{\bw_k w_r} u_{w_r \bw_s \bz_b} u^{\bw_s w_p} u_{w_p
\bw_q z_a} u^{\bw_q w_l} u_{w_l \bga_z}\\
&\ + u_{\ga_z \bw_k} u^{\bw_k w_p} u_{w_p \bw_q z_a} u^{\bw_q w_r} u_{w_r
\bw_s \bz_b} u^{\bw_s w_l} u_{w_l \bga_z} - u_{\ga_z \bw_k} u^{\bw_k w_p} u_{w_p
\bw_q z_a} u^{\bw_q w_l} u_{w_l \bga_z
\bz_b}\\
&\ \left. + u_{\ga_z \bw_k \bz_b} u^{\bw_k w_l} u_{w_l \bga_z z_a} - u_{\ga_z
\bw_k}
u^{\bw_k w_r} u_{w_r \bw_s \bz_b} u^{\bw_s w_l} u_{w_l \bga_z z_a} \right]\\
&\ - u^{\bw_b w_a} \left[ - u_{\ga_z \bw_k \bw_b} u^{\bw_k w_p} u_{w_p \bw_q
w_a}
u^{\bw_q w_l} u_{w_l
\bga_z} + u_{\ga_z \bw_k} u^{\bw_k w_r} u_{w_r \bw_s \bw_b} u^{\bw_s w_p} u_{w_p
\bw_q w_a} u^{\bw_q w_l} u_{w_l \bga_z} \right.\\
&\ \left. + u_{\ga_z \bw_k \bw_b} u^{\bw_k w_l} u_{w_l \bga_z w_a} - u_{\ga_z
\bw_k}
u^{\bw_k w_r} u_{w_r \bw_s \bw_b} u^{\bw_s w_l} u_{w_l \bga_z w_a} \right],
\end{align*}
completing the proof of (\ref{cpmp1}).  Next we establish (\ref{cpmp4}).  Using
(\ref{cpmp210}) we compute
\begin{gather} \label{cpmp410}
\begin{split}
\dt & u_{\ga_z \bw_k} u^{\bw_k \ga_w}\\
=&\ \left( \dt u \right)_{\ga_z \bw_k} u^{\bw_k \ga_w} - u_{\ga_z \bw_k}
u^{\bw_k w_l} \left( \dt u \right)_{w_l \bw_p} u^{\bw_p \ga_w}\\
=&\ \left( u^{\bz_b z_a} u_{z_a \bz_b \ga_z \bw_k} - u^{\bw_b w_a} u_{w_a \bw_b
\ga_z \bw_k} \right.\\
&\ \left. - u^{\bz_b z_a} u^{\bz_d z_c} u_{z_a \bz_d \ga_z} u_{z_c \bz_b \bw_k}
+ u^{\bw_b w_a} u^{\bw_d w_c} u_{w_a \bw_d \ga_z} u_{w_c \bw_b \bw_k} \right)
u^{\bw_k \ga_w}\\
&\ - u_{\ga_z \bw_k} u^{\bw_k w_l} u^{\bw_p \ga_w} \left( u^{\bz_b z_a} u_{z_a
\bz_b w_l \bw_p} - u^{\bw_b w_a} u_{w_a \bw_b w_l \bw_p} \right.\\
&\ \left. - u^{\bz_b z_a} u^{\bz_d z_c} u_{z_a \bz_d w_l} u_{z_c \bz_b \bw_p}  +
u^{\bw_b w_a} u^{\bw_d w_c} u_{w_a \bw_d w_l} u_{w_c \bw_b \bw_p} \right).
\end{split}
\end{gather}
Next we compute partial derivatives
\begin{gather} \label{cpmp420}
\begin{split}
\left(u_{\ga_z \bw_k} u^{\bw_k \ga_w} \right)_{,\mu \rho} =&\ \left(u_{\ga_z
\bw_k \mu} u^{\bw_k \ga_w} - u_{\ga_z \bw_k} u^{\bw_k w_p} u_{w_p \bw_q \mu}
u^{\bw_q \ga_w} \right)_{,\rho}\\
=&\ u_{\ga_z \bw_k \mu \rho} u^{\bw_k \ga_w} - u_{\ga_z \bw_k \mu} u^{\bw_k w_p}
u_{w_p \bw_q \rho} u^{\bw_q \ga_w}\\
&\ - u_{\ga_z \bw_k \rho} u^{\bw_k w_p} u_{w_p \bw_q \mu} u^{\bw_q \ga_w} +
u_{\ga_z \bw_k} u^{\bw_k w_r} u_{w_r \bw_s \rho} u^{\bw_s w_p} u_{w_p \bw_q \mu}
u^{\bw_q \ga_w}\\
&\ - u_{\ga_z \bw_k} u^{\bw_k w_p} u_{w_p \bw_q \mu \rho} u^{\bw_q \ga_w} +
u_{\ga_z \bw_k} u^{\bw_k w_p} u_{w_p \bw_q \mu} u^{\bw_q w_r} u_{w_r \bw_s \rho}
u^{\bw_s \ga_w}.
\end{split}
\end{gather}
Using this we compute
\begin{gather} \label{cpmp430}
\begin{split}
\LL & \left( u_{\ga_z \bw_k} u^{\bw_k \ga_w} \right)\\
=&\ u^{\bz_b z_a} \left( u_{\ga_z \bw_k} u^{\bw_k \ga_w} \right)_{z_a \bz_b} -
u^{\bw_b w_a} \left( u_{\ga_z \bw_k} u^{\bw_k \ga_w} \right)_{w_a \bw_b}\\
=&\ u^{\bz_b z_a} \left[ u_{\ga_z \bw_k z_a \bz_b} u^{\bw_k \ga_w} - u_{\ga_z
\bw_k z_a} u^{\bw_k w_p} u_{w_p \bw_q \bz_b} u^{\bw_q \ga_w} \right.\\
&\ - u_{\ga_z \bw_k \bz_b} u^{\bw_k w_p} u_{w_p \bw_q z_a} u^{\bw_q \ga_w} +
u_{\ga_z \bw_k} u^{\bw_k w_r} u_{w_r \bw_s \bz_b} u^{\bw_s w_p} u_{w_p \bw_q
z_a} u^{\bw_q \ga_w}\\
&\ \left. - u_{\ga_z \bw_k} u^{\bw_k w_p} u_{w_p \bw_q z_a \bz_b} u^{\bw_q
\ga_w} + u_{\ga_z \bw_k} u^{\bw_k w_p} u_{w_p \bw_q z_a} u^{\bw_q w_r} u_{w_r
\bw_s \bz_b} u^{\bw_s \ga_w} \right]\\
&\ - u^{\bw_b w_a} \left[ u_{\ga_z \bw_k w_a \bw_b} u^{\bw_k \ga_w} - u_{\ga_z
\bw_k w_a} u^{\bw_k w_p} u_{w_p \bw_q \bw_b} u^{\bw_q \ga_w} \right.\\
&\ - u_{\ga_z \bw_k \bw_b} u^{\bw_k w_p} u_{w_p \bw_q w_a} u^{\bw_q \ga_w} +
u_{\ga_z \bw_k} u^{\bw_k w_r} u_{w_r \bw_s \bw_b} u^{\bw_s w_p} u_{w_p \bw_q
w_a} u^{\bw_q \ga_w}\\
&\ \left. - u_{\ga_z \bw_k} u^{\bw_k w_p} u_{w_p \bw_q w_a \bw_b} u^{\bw_q
\ga_w} + u_{\ga_z \bw_k} u^{\bw_k w_p} u_{w_p \bw_q w_a} u^{\bw_q w_r} u_{w_r
\bw_s \bw_b} u^{\bw_s \ga_w} \right].
\end{split}
\end{gather}
Combining (\ref{cpmp410}) and (\ref{cpmp430}) yields
\begin{align*}
& \left( \dt - \LL \right) W_{\ga_z \ga_w}\\
=&\ \left( - u^{\bz_b z_a} u^{\bz_d z_c} u_{z_a \bz_d \ga_z} u_{z_c \bz_b \bw_k}
+ u^{\bw_b w_a} u^{\bw_d w_c} u_{w_a \bw_d \ga_z} u_{w_c \bw_b \bw_k} \right)
u^{\bw_k \ga_w}\\
&\ - u_{\ga_z \bw_k} u^{\bw_k w_l} u^{\bw_p \ga_w} \left( - u^{\bz_b z_a}
u^{\bz_d z_c} u_{z_a \bz_d w_l} u_{z_c \bz_b \bw_p}  + u^{\bw_b w_a} u^{\bw_d
w_c} u_{w_a \bw_d w_l} u_{w_c \bw_b \bw_p} \right)\\
&\ - u^{\bz_b z_a} \left[  - u_{\ga_z \bw_k z_a} u^{\bw_k w_p} u_{w_p \bw_q
\bz_b} u^{\bw_q \ga_w}  - u_{\ga_z \bw_k \bz_b} u^{\bw_k w_p} u_{w_p \bw_q z_a}
u^{\bw_q \ga_w} \right.\\
&\ \left. + u_{\ga_z \bw_k} u^{\bw_k w_r} u_{w_r \bw_s \bz_b} u^{\bw_s w_p}
u_{w_p \bw_q z_a} u^{\bw_q \ga_w} + u_{\ga_z \bw_k} u^{\bw_k w_p} u_{w_p \bw_q
z_a} u^{\bw_q w_r} u_{w_r \bw_s \bz_b} u^{\bw_s \ga_w} \right]\\
&\ + u^{\bw_b w_a} \left[  - u_{\ga_z \bw_k w_a} u^{\bw_k w_p} u_{w_p \bw_q
\bw_b} u^{\bw_q \ga_w}  - u_{\ga_z \bw_k \bw_b} u^{\bw_k w_p} u_{w_p \bw_q w_a}
u^{\bw_q \ga_w} \right.\\
&\ \left. + u_{\ga_z \bw_k} u^{\bw_k w_r} u_{w_r \bw_s \bw_b} u^{\bw_s w_p}
u_{w_p \bw_q w_a} u^{\bw_q \ga_w} + u_{\ga_z \bw_k} u^{\bw_k w_p} u_{w_p \bw_q
w_a} u^{\bw_q w_r} u_{w_r \bw_s \bw_b} u^{\bw_s \ga_w} \right]\\
=:&\ \sum_{i=1}^{12} A_i.
\end{align*}
Observing that $W$ is Hermitian, and that the operator $\dt - \LL $ is Hermitian
 we obtain  (\ref{cpmp5}). 
\end{proof}
\end{lemma}

\begin{lemma} \label{complexQsign} With the setup above,
\begin{align*}
Q \leq 0.
\end{align*}
\begin{proof} Using Lemma \ref{complexevs} we compute
\begin{align*}
Q(\ga,\bga) =&\ Q_{\ga_z \bga_z} + Q_{\ga_z \bga_w} + Q_{\ga_w \bga_z} +
Q_{\ga_w \bga_w}\\
=&\ - u^{\bz_q z_r} u^{\bz_s z_p} u_{z_p
\bz_q \ga_z} u_{z_r \bz_s \bga_z} + u^{\bz_q z_r} u^{\bz_s z_p}
u_{z_p \bz_q \ga_z} u_{z_r \bz_s \bw_k} u^{\bw_k w_l} u_{w_l
\bga_z}\\
&\ - u_{\ga_z \bw_k} u^{\bw_k w_p} u^{\bz_q z_r} u^{\bz_s z_p} u_{z_p \bz_q w_p}
u_{z_r \bz_s \bw_q}
u^{\bw_q w_l} u_{w_l \bga_z} + u_{\ga_z \bw_k} u^{\bw_k w_l} 
u^{\bz_q z_r} u^{\bz_s z_p} u_{z_p \bz_q w_l} u_{z_r \bz_s \bga_z}\\
&\ + u^{\bz_b z_a} \left[ -
u_{\ga_z \bw_k z_a} u^{\bw_k w_p} u_{w_p \bw_q \bz_b} u^{\bw_q w_l} u_{w_l
\bga_z}
+ u_{\ga_z \bw_k z_a} u^{\bw_k w_l} u_{w_l \bga_z \bz_b} \right.\\
&\ - u_{\ga_z \bw_k \bz_b} u^{\bw_k w_p} u_{w_p \bw_q z_a} u^{\bw_q w_l} u_{w_l
\bga_z} + u_{\ga_z \bw_k} u^{\bw_k w_r} u_{w_r \bw_s \bz_b} u^{\bw_s w_p} u_{w_p
\bw_q z_a} u^{\bw_q w_l} u_{w_l \bga_z}\\
&\ + u_{\ga_z \bw_k} u^{\bw_k w_p} u_{w_p \bw_q z_a} u^{\bw_q w_r} u_{w_r
\bw_s \bz_b} u^{\bw_s w_l} u_{w_l \bga_z} - u_{\ga_z \bw_k} u^{\bw_k w_p} u_{w_p
\bw_q z_a} u^{\bw_q w_l} u_{w_l \bga_z
\bz_b}\\
&\ \left. + u_{\ga_z \bw_k \bz_b} u^{\bw_k w_l} u_{w_l \bga_z z_a} - u_{\ga_z
\bw_k}
u^{\bw_k w_r} u_{w_r \bw_s \bz_b} u^{\bw_s w_l} u_{w_l \bga_z z_a}
\right]\\
&\ - u^{\bw_b w_a} \left[ - u_{\ga_z \bw_k \bw_b} u^{\bw_k w_p} u_{w_p \bw_q
w_a}
u^{\bw_q w_l} u_{w_l
\bga_z} + u_{\ga_z \bw_k} u^{\bw_k w_r} u_{w_r \bw_s \bw_b} u^{\bw_s w_p} u_{w_p
\bw_q w_a} u^{\bw_q w_l} u_{w_l \bga_z} \right.\\
&\ \left. + u_{\ga_z \bw_k \bw_b} u^{\bw_k w_l} u_{w_l \bga_z w_a} - u_{\ga_z
\bw_k}
u^{\bw_k w_r} u_{w_r \bw_s \bw_b} u^{\bw_s w_l} u_{w_l \bga_z w_a} \right]\\
&\ - u^{\bga_w w_k} u^{\bw_l \ga_w} u^{\bz_q z_r} u^{\bz_s z_p} u_{z_p \bz_q
w_k}
u_{z_r \bz_s \bw_l} + u^{\bz_l z_k} u^{\bga_w w_p} u_{w_p \bw_q \bz_l} u^{\bw_q
w_j} u_{w_j
\bw_k z_k} u^{\bw_k \ga_w}\\
&\ + u^{\bz_l z_k} u^{\bga_w w_j} u_{w_j \bw_k z_k} u^{\bw_k w_p} u_{w_p
\bw_q \bz_l} u^{\bw_q \ga_w}  - u^{\bw_l w_k} u^{\bga_w w_p} u_{w_p \bw_q \bw_l}
u^{\bw_q w_j} u_{w_j
\bw_r w_k} u^{\bw_r \ga_w}\\
&\ - u^{\bz_b z_a} u^{\bz_d z_c} u_{z_a \bz_d \ga_z} u_{z_c \bz_b \bw_k}
u^{\bw_k \ga_w} + u_{\ga_z \bw_k} u^{\bw_k w_l} u^{\bw_p \ga_w}  u^{\bz_b z_a}
u^{\bz_d z_c} u_{z_a \bz_d w_l} u_{z_c \bz_b \bw_p}\\
&\ - u^{\bz_b z_a} \left[  - u_{\ga_z \bw_k z_a} u^{\bw_k w_p} u_{w_p \bw_q
\bz_b} u^{\bw_q \ga_w}  - u_{\ga_z \bw_k \bz_b} u^{\bw_k w_p} u_{w_p \bw_q z_a}
u^{\bw_q \ga_w} \right.\\
&\ \left. + u_{\ga_z \bw_k} u^{\bw_k w_r} u_{w_r \bw_s \bz_b} u^{\bw_s w_p}
u_{w_p \bw_q z_a} u^{\bw_q \ga_w} + u_{\ga_z \bw_k} u^{\bw_k w_p} u_{w_p \bw_q
z_a} u^{\bw_q w_r} u_{w_r \bw_s \bz_b} u^{\bw_s \ga_w} \right]\\
&\ + u^{\bw_b w_a} \left[  - u_{\ga_z \bw_k \bw_b} u^{\bw_k w_p} u_{w_p \bw_q
w_a} u^{\bw_q \ga_w} + u_{\ga_z \bw_k} u^{\bw_k w_r} u_{w_r \bw_s \bw_b}
u^{\bw_s w_p} u_{w_p \bw_q w_a} u^{\bw_q \ga_w}\right]\\
&\ + u^{\bga_w w_p} u^{\bw_q w_k} u_{w_k \bga_z}  u^{\bz_b z_a} u^{\bz_d z_c}
u_{z_a \bz_d w_p} u_{z_c \bz_b \bw_q}  - u^{\bga_w w_k} u^{\bz_b z_a} u^{\bz_d
z_c} u_{z_a \bz_d w_k} u_{z_c \bz_b \bga_z}\\
&\ - u^{\bz_b z_a} \left[ u^{\bga_w w_r} u_{w_r \bw_s \bz_b} u^{\bw_s w_p}
u_{w_p \bw_q z_a} u^{\bw_q w_k} u_{w_k \bga_z}  + u^{\bga_w w_p} u_{w_p \bw_q
z_a} u^{\bw_q w_r} u_{w_r \bw_s \bz_b} u^{\bw_s w_k} u_{w_k \bga_z} \right.\\
&\ \left. - u^{\bga_w w_p} u_{w_p \bw_q z_a} u^{\bw_q w_k} u_{w_k \bga_z \bz_b}
- u^{\bga_w w_p} u_{w_p \bw_q \bz_b} u^{\bw_q w_k} u_{w_k \bga_z z_a}\right]\\
&\ + u^{\bw_b w_a} \left[ u^{\bga_w w_r} u_{w_r \bw_s \bw_b} u^{\bw_s w_p}
u_{w_p \bw_q w_a} u^{\bw_q w_k} u_{w_k \bga_z} - u^{\bga_w w_p} u_{w_p \bw_q
\bw_b} u^{\bw_q w_k} u_{w_k \bga_z w_a} \right]\\
=:&\ \sum_{i=1}^{36} A_i.
\end{align*}
We observe:
\begin{align*}
 A_{21} + A_{22}+A_{29} + A_{30} =& \Re \left[ \ u^{\bz_b z_a} u^{\bz_d z_c}
\left( u_{z_c \bz_b \bw_p}
u^{\bw_p \ga_w} \right) \left(u_{z_a \bz_d w_l} u^{\bw_k w_l} u_{\ga_z \bw_k} -
u_{\ga_z z_a \bz_d} \right) \right]\\
 \leq&\ -  \left(  A_1 + A_2 + A_3 + A_{4} + A_{17}  \right).
\end{align*} using Cauchy-Schwarz.  
Next
\begin{align*}
 A_{23} + A_{26} + A_{32}+ A_{33}=&\Re \left[ \ u^{\bz_b z_a} u^{\bw_k w_p}
\left( u_{w_p \bw_q \bz_b}
u^{\bw_q \ga_w} \right) \left( u_{\ga_z z_a \bw_k} - u_{z_a w_p \bw_k} u^{\bw_q
w_p} u_{\ga_z \bw_q} \right) \right]\\
\leq&\ - \left( A_5 + A_6 + A_9 + A_{10} + A_{18} \right).
\end{align*}
Next
\begin{align*}
 A_{24} + A_{25} + A_{31}+ A_{34} =& \Re \left[ \ u^{\bz_b z_a} u^{\bw_k w_p}
\left( u_{w_p \bw_q z_a}
u^{\bw_q \ga_w} \right) \left( u_{\ga_z \bw_k \bz_b} - u_{\ga_z \bw_s} u^{\bw_s
w_r} u_{\bz_b w_r \bw_k} \right) \right]\\
\leq&\ - \left( A_7 + A_8 + A_{11} + A_{12} + A_{19} \right).
\end{align*}
Next
\begin{align*}
A_{27} + A_{28} + A_{35}+ A_{36} =& \Re \left[ \ u^{\bw_b w_a} u^{\bw_k w_p}
\left( u_{w_p \bw_q w_a}
u^{\bw_q \ga_w} \right) \left( u_{\ga_z \bw_s} u^{\bw_s w_r} u_{w_r \bw_k \bw_b}
- u_{\ga_z \bw_k\bw_b} \right) \right]\\
 \leq&\ -  \left(  A_{13}+A_{14}+A_{15}+A_{16} + A_{20} \right).
\end{align*}

\end{proof}
\end{lemma}

\begin{lemma} \label{Wev} Let $u_t$ be a solution to (\ref{realparabolic}) such
that $u_t \in
\mathcal E$ for all $t$.  Then
\begin{gather*}
\begin{split} 
\left( \dt - \LL \right) \frac{\del u}{\del t}=&\ 0.
\end{split}
\end{gather*}
Also,
\begin{align*}
\left( \dt - \LL \right) W \leq 0.
\end{align*}
\begin{proof} Let $u_t$ be as in the statement.  Define $v_t : \mathbb C^n \to
\mathbb R$, by
\begin{align*}
v_t(z_1,\dots,z_n) = u_t(\Real z_1,\dots, \Real z_n).
\end{align*}
Elementary calculations show that $v_t \in \EE$ and that $v_t$ is a solution to
(\ref{complexparabolic}).  Moreover the matrix $W$ associated to $v_t$ via
(\ref{cW}) agrees with the matrix $\N^2 w$ as in (\ref{twistedhessian}).  The
result follows from Lemmas \ref{complexevs} and \ref{complexQsign}.
\end{proof}
\end{lemma}

\bibliographystyle{hamsplain}

\end{document}